\documentclass{article}

\usepackage[letterpaper,top=2cm,bottom=2cm,left=3cm,right=3cm,marginparwidth=1.75cm]{geometry}
\usepackage{graphicx}
\usepackage[colorlinks=true, allcolors=blue]{hyperref}
\usepackage[tbtags]{amsmath}
\usepackage{amsfonts,amssymb,mathrsfs,amscd}

\numberwithin{equation}{section}
\newtheorem{theorem}{Theorem}[section]
\newtheorem{lemma}[theorem]{Lemma}
\newtheorem{corollary}[theorem]{Corollary}
\newtheorem{propos}[theorem]{Proposition}
\newtheorem{definition}[theorem]{Definition}
\newtheorem{proof}{Proof}
\newtheorem{remark}[theorem]{Remark}
\newtheorem{example}[theorem]{Example}

\def\Im{\operatorname{Im}}

\def\supp{\operatorname{supp}}
\newcommand{\ABK}{\mathscr {A\!\!B\!\!K}}

\title{Method of similar operators in harmonious Banach spaces}
\author{Anatoly G.~Baskakov, Ilya A.~Krishtal, and Natalia B.~Uskova}

\begin{document}

\maketitle

\begin{abstract}
We consider similarity transformations of a perturbed linear operator
$A-B$ in a complex Banach space $\mathcal{X}$, where the unperturbed operator $A$
is a generator of a Banach $L_1(\mathbb{R})$-module and the perturbation operator $B$ is a bounded linear operator. The result of the transformation is a simpler operator $A-B_0$. For example, if $A$ is a differentiation operator and $B$ is an operator of multiplication by an operator-valued function, then
 $B_0$ is an operator of multiplication by a function that is a restriction of an entire function of exponential type and could be $0$ in some cases. As a consequence, we derive the spectral invariance of the operator $A-B$ in a large class of spaces. The study is based on a widely applicable modification of the method of similar operators that is also presented in the paper. This non-traditional modification is rooted in the spectral theory of Banach $L_1(\mathbb{R})$-modules.
\end{abstract}

\markright{Method of similar operators in harmonious Banach spaces}

\section{Introduction}

Similarity transformations constitute a well-established technique that has enjoyed enduring relevance across the expanse of mathematical exploration.
In the realm of  perturbation theory \cite{BS93, DS88III, Ka76, Sh16}, they find their main application in the study of spectral properties of linear operators in a diverse array of classes (see, e.~g., \cite{B16, C85,  F65p, K18, KS18, M06, SYY20, S83}). Various methods utilizing similarity transformations  have a long history (see, e.~g., \cite{C85, DL57, F65p}) and an abundant assortment of adaptations tailored to address specific contexts (see \cite[and reference therein]{B87, B95, BDS11, KS18, SYY20, U00}). One of these methods within the purview of perturbation theory can be extremely powerful but is currently often overlooked by the researchers in the field. It is commonly referred to as the method of similar operators and in this semi-expository paper, we describe its contemporary advancements in a fairly general setting of Banach $L_1(\mathbb{R})$-modules. This setting, drawing on a multitude of tools of functional and harmonic analysis, allows one to create a comprehensive theory that can be used in a vast variety of applications.

 It is noteworthy that in narrower contexts, such as that of Hilbert spaces, modifications of the method of similar operators can yield more precise outcomes (as demonstrated, for example, in \cite{BDS11, BKU19, BP16, BP17}). In other specialized settings, different perturbation theory methods for analyzing the spectral properties of linear operators may offer results of comparable or heightened quality \cite{Ka67, M88, MM82, MM84, MS17, SD77, SD79, Sh16}. 
 The core emphasis of this paper lies in the presentation of a version of the method of similar operators that strikes a balance between generality and adaptability, rendering it readily applicable in a wide array of important scenarios.
 
 Throughout this exposition, we present an improvement of some of the established  results within the method (illustrated by, for instance, Theorem \ref{baskth3.1}), offer more streamlined proofs for select applications (exemplified in Theorem \ref{hypcaussim}), and obtain new results (such as Corollary \ref{maincor} and Theorem \ref{baskth6.1}). 
 
 To offer a taste of the outcomes of the application of the method of similar operators appearing below, we mention the following two results. Theorem \ref{baskth12} can be interpreted as a method of reduction of certain non-autonomous abstract initial value problems to autonomous ones. Theorem \ref{mathrth3} establishes similarity of certain unbounded triangular matrices to their main diagonal. It is remarkable that these two theorems follow from essentially the same more general result.
 
Let $\mathcal{X}$ be a complex Banach space and  $\mathrm{End}\,\mathcal{X}$ be a Banach algebra of all bounded linear operators in $\mathcal{X}$. We consider a closed linear operator $A: D(A)\subset\mathcal{X}\to\mathcal{X}$ such that the operator $iA$ is a generator of a strongly continuous isometric operator semigroup and the linear operator $B$ is bounded, that is,
$B\in\mathrm{End}\,\mathcal{X}$. We construct a similarity transformation of the operator  $A-B$ into a simpler operator $A-B_0$, where the operator $B_0\in\mathrm{End}\,\mathcal{X}$ 
has a compact Beurling spectrum (see Definition \ref{baskdef2.7}). The latter property is desirable in applications as it opens up the possibility of using a much larger collection of numerical approximation methods.

It is to obtain the operator $B_0$, that we use the method of similar operators that we advocate. The standard versions of the method can be found, for example, in \cite{B86, BDS11, BKU19, BP17}. In this paper, however, we present a non-traditional modification of the method, which is applicable more generally than the standard ones and is based in part on the work in \cite{B87, B94}.

Let us emphasize that the modification of the method of similar operators used in this paper is considerably different from the standard scheme in \cite{B86,  BDS11, BKU19, BP17} and it improves  the results in \cite{B87, B94} that are closest to it. The key difference, in comparison with the standard scheme, is that one of the main transforms (operator $J$) used to construct the similarity transformation is no longer assumed to be an idempotent.
Consequently, the main equation of the method of similar operator \eqref{bask3.1} differs from the standard one \cite[Eq.~(3.3)]{B86}. Condition \eqref{bask3.2} of the main abstract similarity theorem (Theorem~\ref{baskth3.1}) also differs from its standard analog  (\cite[Theorem~3.1, Eq.~(3.4)]{B86}). We also remark that in the traditional scheme the unperturbed operator $A$ is typically assumed to have discrete spectrum or at least have separated spectral components
(\cite{B86, B87, B94, BDS11, BKU19, BP17}). The scheme developed in this paper is equipped to handle the case with no gaps in the spectrum, i.~e.~$\sigma(A)=\mathbb R$.

Moreover, in this paper, we consider unperturbed operators $A$ in a larger class than mentioned above. We assume that the operator  $A$ is a generator of a non-degenerate Banach $L_1(\mathbb{R})$-module with the structure associated with an isometric representation  $\mathcal T: \mathbb{R}\to\mathrm{End}\,\mathcal{X}$; 
the operator $B$  is assumed to belong to a certain class
$\mathcal{M}_0\subseteq\mathrm{End}\,\mathcal{X}$ which coincides with $\mathrm{End}\,\mathcal{X}$ if the representation  $\mathcal T$ is strongly continuous (see Section \ref{basksec2} for the relevant definitions).

The establishment of the similarity theorems heavily relies on the spectral theory of Banach $L_1(\mathbb{R})$-modules. The indispensable components of this framework are presented in Section \ref{basksec2} using a specific  customization for the subsequent proofs.

In the following Section \ref{HarmSp}, we introduce the notion of spaces that are harmonious to the operator $A$ (see Definition~\ref{baskdef16}) and give illustrative examples of such spaces. In particular, we define $A$-harmonious spectral submodules, submodules of periodic and almost periodic operators, Wiener classes, Beurling classes, Jaffard classes, as well as $\ABK$-classes. Some of these spaces (under additional restrictions) are then used as admissible perturbation spaces (see Definition \ref{baskdef3.3}).   

In Section \ref{basksec3}, we present an abstract scheme of the method of similar operators that takes into account that the transform  $J$ does not have to be an idempotent. Depending on the assumptions, we prove four different similarity theorems: Theorem \ref{baskth3.1} covers the most general case; Theorem \ref{baskth3.2} applies when  $J$ is an idempotent (this theorem is most commonly used in the literature, \cite{B86, B87, B94, BDS11, BKU19, BP17});
Theorem~\ref{baskth3.2c} is a modification of Theorem \ref{baskth3.2}, which includes an additional condition on the perturbation $B$: 
it is assumed that $JB=0$; finally, Theorem~\ref{baskth3.4} covers the special case of Theorem~\ref{baskth3.2c} in which $J=0$ (in this case, the main non-linear equation of the method reduces to the linear equation studied by Friedrichs 
\cite{DS88III, F65p}). We purposefully  collected the above known and new results on the method of similar operators to create a convenient resource for future reference; to the best of our knowledge such a classification of various theorems of the method of similar operators does not exist elsewhere in the literature. 

Section \ref{basksec4} is the centerpiece of the paper. Here we showcase how the abstract scheme of the method of similar operators performs in the setting of harmonious spaces (see Definition~\ref{baskdef16}). In particular, we present a specific construction of admissible triples (as proved in Theorem \ref{baskth4.1}), which allows us to establish  
 similarity of various perturbed operators in Theorem \ref{baskth4.2}. Leveraging this similarity, we proceed to prove a result about spectral invariance in Theorem \ref{baskth4.3}. Additionally, in Theorem \ref{baskth5.1}, we weaken conditions for similarity in the special case when $0$ is an isolated point in the Beurling spectrum of the perturbation operator $B$. We conclude the section with the case when operator $B$ is hypercausal  (see Definition~\ref{baskdef2.9}). In this case, $J=0$ and we initially present Theorem~\ref{baskth5.2}, which is the result of using Friedrichs' method (Theorem~\ref{baskth3.4}) in this setting. It turns out, however, that hypercausality allows one to prove a much stronger result, which requires no conditions on the norm of the perturbation (see Theorem \ref{hypcaussim}).

In Sections \ref{basksec5} and \ref{basksec6}, the theorems of
Section \ref{basksec4} are applied for specific classes of operators. In particular, in Section \ref{basksec5}, we consider first order differential operators in homogeneous function spaces (see Definition \ref{baskdef26}) with potentials from $\mathcal{M}$-compatible function spaces
 (see Definition \ref{baskdef27}). We show that such an operator is similar to a differential operator with a potential that is a restriction of an entire function of exponential type (see Theorem \ref{baskth6.1}). 
Additionally, we provide sufficient conditions for similarity of the original operator to an operator with a constant potential (Theorem \ref{baskth12}), and to the unperturbed differentiation operator itself (Theorem \ref{baskth13}). 
In Section \ref{basksec6}, we apply the main theorems to operators defined by their matrices with respect to a disjunctive resolution of the identity. We use these examples to also illustrate the connections and the differences of the scheme of the method of similar operators used in this paper and the standard scheme of    \cite{B86, B87, B94, BDS11, BKU19, BP17} that is  commonly used for operator matrices.

Let us conclude the introduction with a description of standard notions and notation used throughout the paper. 

By $\sigma(A)$ and $\rho(A)$ we shall denote, respectively,  the spectrum and the resolvent set of a linear operator $A$. By the resolvent of $A$ we shall mean the operator-valued function $R = R(\cdot\,; A)$: $\rho(A)\to \mathrm{End}\,\mathcal{X}$ given by $R(\lambda) = (A-\lambda I)^{-1}$.

By $L_p(\mathbb{R}, \mathscr{X})$, $p\in [1, \infty]$, we shall denote the standard Bochner-Lebesgue spaces. Recall that
$L_p(\mathbb{R}, \mathscr{X})$ is a Banach space of (equivalence classes of) Bochner measurable $\mathscr X$-valued functions
on $\mathbb{R}$ that are $p$-summable if $p\in[1,\infty)$ or essentially bounded if  $p = \infty$. The norms in $L_p(\mathbb{R}, \mathscr{X})$, $p\in [1, \infty]$, are given by
$$
\|x\|_{p}=\Big(\int_{\mathbb{R}}\|x(t)\|_{\mathscr{X}}^p\,dt\Big)^{1/p}, \quad p\in [1, \infty),
$$
and $\|x\|_{\infty}=\mathrm{ess}\,\sup_{t\in\mathbb{R}}\|x(t)\|_{\mathscr{X}}$.
The classical Lebesgue spaces are, of course, a special case of the Bochner-Lebesgue spaces. We will commonly use an abbreviated notation for them; in particular, for
$L_1=L_1(\mathbb{R})=L_1(\mathbb{R}, \mathbb{C})$ and $L_2=L_2(\mathbb{R})=L_2(\mathbb{R}, \mathbb{C})$. Recall that $L_1$ is a Banach algebra with the multiplication operator given by the convolution 
$$
(f*g)(t)=\int_{\mathbb{R}}f(\tau)g(t-\tau)\,d\tau, \quad f, g\in L_1(\mathbb{R}).
$$
Recall also that $L_2$ is a Hilbert space with the inner product given by
$$
\langle f, g \rangle =\int_{\mathbb{R}}f(t)\overline{g(t)}\,dt, \quad f, g\in L_2(\mathbb{R}).
$$

We will use the Fourier transform of a function $f\in L_1$ given by
$$
\widehat{f}(\lambda)=\int_{\mathbb{R}}f(t)e^{-i\lambda t}\,dt, \quad \lambda\in\mathbb{R}, \quad f\in L_1,
$$
and its standard extension to $L_2$. By $\widehat{L}_1=\widehat{L}_1(\mathbb{R})$
we shall denote the Fourier algebra of functions that is isomorphic (as a Banach algebra) to  $L_1$ via the Fourier transform. The multiplication operation in $\widehat{L}_1$ is the pointwise multiplication of functions and the norm is
$$
\|\widehat{f}\|_\infty=\max_{\lambda\in\mathbb{R}}|\widehat{f}(\lambda)|, \quad f\in L_1.
$$
For $f\in  L_2$ we have the standard Parseval equality $\|\widehat{f}\|_2=\sqrt{2\pi}\|f\|_2$.
 
\section{Banach modules}\label{basksec2}

In this section, we outline the basic notions and results of the spectral theory of Banach modules that are necessary for the subsequent exposition.
We follow \cite{B04, BK05, BK14, BKU20, L53, RS00} in our presentation.

\begin{definition}
A complex Banach space $\mathcal{X}$ is a Banach $L_1(\mathbb{R})$-module if there is a bilinear map
 $(f, x)\mapsto fx: L_1(\mathbb{R})\times\mathcal{X}\to\mathcal{X}$, with the properties:

\begin{enumerate}
    \item $(a*b)x=a(bx)$, $a$, $b\in L_1$, $x\in\mathcal{X}$;
    \item $\|ax\|\leq \|a\|_1\|x\|$, $a\in L_1$, $x\in\mathcal{X}$.
\end{enumerate}
\end{definition}

\begin{definition}\label{def2}
A Banach module is non-degenerate if   $fx=0$ for all $f\in L_1$ implies that $x=0$.
\end{definition}

The structure of a Banach  $L_1(\mathbb{R})$-module is often associated with a bounded representation of the group $\mathbb{R}$.

\begin{definition}\label{baskdef2.3}
A map $\mathcal{T}: \mathbb{R}\to\mathrm{End}\,\mathcal{X}$ is called a representation of the group $\mathbb{R}$
by operators from $\mathrm{End}\,\mathcal{X}$ if $\mathcal{T}(0)=I$ and $\mathcal{T}(t+s)=\mathcal{T}(t)\mathcal{T}(s)$, $t, s\in\mathbb{R}$. A representation  $\mathcal{T}$ is strongly continuous if every function of the form
$\tau_x: \mathbb{R}\to\mathrm{End}\,\mathcal{X}$, $\tau_x(t)=\mathcal{T}(t)x$, $x\in\mathcal{X}$, is continuous. A representation
$\mathcal{T}: \mathbb{R}\to\mathrm{End}\,\mathcal{X}$ is isometric if
$\|\mathcal{T}(t)x\|=\|x\|$ for all $t\in\mathbb{R}$ and $x\in\mathcal{X}$.
\end{definition}

Let us give a few examples of most commonly used representations.

\begin{example}\label{exrepSM}
Let $\mathcal X = L_p(\mathbb{R}, {\mathscr{X}})$, $p\in [1, \infty]$.
Then $S, M: \mathbb{R}\to\mathrm{End}\,\mathcal X$,
\begin{equation}\label{basknew1}
(S(s)x)(t)=x(s+t), \quad t, s\in\mathbb{R}, \quad x\in \mathcal X,
\end{equation}
\begin{equation}\label{basknew1'}
(M(s)x)(t)=e^{ist}x(t), \quad t, s\in\mathbb{R}, \quad x\in \mathcal X,
\end{equation}
are isometric representations of the group $\mathbb{R}$ by operators from $\mathrm{End}\, \mathcal X$. They are strongly continuous if $p\neq  \infty$.
The representation $S$ will be referred to as {translation} and $M$ -- as {modulation}.
\end{example}
We shall exhibit one more important example of a representation below (see Example \ref{exrepU}).

\begin{definition}\label{baskdef2.4}
We say that a module structure on $\mathcal{X}$ is associated with a representation 
$\mathcal{T}: \mathbb{R}\to\mathrm{End}\,\mathcal{X}$ if for any   $x\in\mathcal{X}$,
$t\in\mathbb{R}$ and $f\in L_1(\mathbb{R})$ we have
$$
\mathcal{T}(t)(fx)=(S(t)f)x=f(\mathcal{T}(t)x),
$$
where the representation  $S: \mathbb{R}\to\mathrm{End}\,L_1$ is given by \eqref{basknew1}.
\end{definition}

In order to emphasize that the module structure on $\mathcal{X}$ is associated with a representation 
$\mathcal{T}$, we may use the notation $(\mathcal{X}, \mathcal{T})$ instead of $\mathcal{X}$. We shall also 
occasionally employ a slight abuse of notation by writing $\mathcal{T}(f)x$  instead of $fx$ for $f\in L_1$ and $x\in(\mathcal{X}, \mathcal{T})$.
This is justified via identifying $\mathbb{R}$ with Dirac measures, $L_1$  -- with absolutely continuous measures (with respect to the Lebesgue measure), and viewing $\mathcal{T}$ as a representation of the algebra of measures.

\begin{definition}\label{defharm}
We call a non-degenerate Banach  $L_1(\mathbb{R})$-module $(\mathcal{X}, \mathcal{T})$  a harmonious space.
\end{definition}

\begin{remark}\label{ogriz}
Hereinafter, we consider only non-degenerate Banach  $L_1(\mathbb{R})$-modules and all representations are assumed to be isometric; these two assumptions may not be explicitly stated. We remark that if a representation is merely bounded, then the space   $\mathcal X$ admits an equivalent norm $|||\cdot|||$, 
\[
|||x||| = \sup_{t\in\mathbb R} \|\mathcal T(t)x\|, \ x\in\mathcal X,
\]
which turns the representation into an isometric one.
\end{remark}

Let $\mathcal{T}: \mathbb{R}\to\mathrm{End}\,\mathcal{X}$ be a strongly continuous representation. Then the formula
\begin{equation}\label{basknew2}
\mathcal T(f)x=fx=\int_{\mathbb{R}}f(t)\mathcal{T}(-t)x\,dt, \quad f\in L_1, \quad x\in\mathcal{X},
\end{equation}
defines a Banach   $L_1(\mathbb{R})$-module structure on $\mathcal{X}$ that is associated with the representation  $\mathcal{T}$.
We note that  if $(\mathcal{X}, \mathcal{T}_1)$ and $(\mathcal{X}, \mathcal{T}_2)$ are non-degenerate Banach modules with the same structure then $\mathcal{T}_1=\mathcal{T}_2$ (see \cite[Lemma 2.2]{BK05}).

\begin{definition}
Let $(\mathcal{X}, \mathcal{T})$ be a Banach $L_1(\mathbb{R})$-module. A linear subspace
$\mathcal{F}\subset\mathcal{X}$ is a submodule if $\mathcal T(t)x \in \mathcal F$ and $ax\in\mathcal{F}$ for all $t\in\mathbb R$, $a\in L_1$ and $x\in\mathcal{F}$.
\end{definition}
Observe that any closed submodule  $\mathcal{F}$ is itself a Banach module.

\begin{definition}
A vector $x$  in a Banach $L_1(\mathbb{R})$-module $(\mathcal{X}, \mathcal{T})$ is called $\mathcal{T}$-continuous if the function $\tau_x: \mathbb{R}\to\mathcal{X}$ from Definition \ref{baskdef2.3} is continuous.
\end{definition}

The set of all $\mathcal{T}$-continuous vectors from $\mathcal{X}$ is a closed submodule of $\mathcal{X}$, which we will denote by $\mathcal{X}_c$.
For $x\in\mathcal{X}_c$ and $f\in L_1(\mathbb{R})$, we have  that formula \eqref{basknew2} holds (see \cite{B04}).

\begin{definition}\label{baskdef2.7}
Let $(\mathcal{X}, \mathcal{T})$ be a Banach $L_1(\mathbb{R})$-module and $Y$ be a subspace of $\mathcal{X}$.
The Beurling spectrum  $\Lambda(Y)=\Lambda(Y, \mathcal{T})$ of $Y$ is defined by
$$
\Lambda(Y, \mathcal{T})=\{\lambda\in\mathbb{R}:\; \mbox{if}\; fx=0 \; \mbox{ for some }\; f\in L_1 \mbox{ and all }\; x\in Y \;
\mbox{ then }\; \widehat{f}(\lambda)=0\}.
$$
\end{definition}

If $Y$ consists of a single vector $x\in\mathcal{X}$, then we will write $\Lambda(x)$ instead of  $\Lambda(Y) = \Lambda(\{x\})$. Observe that
$$
\Lambda(x)=\{\lambda\in\mathbb{R}:\;fx\ne 0 \;\mbox{ for any }\; f\in L_1 \;\mbox{ such that }\; \widehat{f}(\lambda)\ne 0\}.
$$

The basic properties of the Beurling spectrum are summarized in the following lemma.

\begin{lemma}[\cite{BK05}]\label{basklh2.1new}
Let $(\mathcal{X}, \mathcal{T})$ be a Banach $L_1(\mathbb{R})$-module. The following properties hold.
\begin{enumerate}
    \item $\Lambda(Y)$ is closed for any $Y\subseteq\mathcal{X}$.
    \item $\Lambda(Y)=\varnothing$ if and only if $Y=\{0\}$.
    \item $\Lambda(Ax+By)\subseteq \Lambda(x)\cup\Lambda(y)$ for any  $A, B\in\mathrm{End}\,\mathcal{X}$ that commute with all $\mathcal{T}(f)$, $f\in L_1(\mathbb{R})$.
    \item $\Lambda(fx)\subseteq (\mathrm{supp}\,\widehat{f})\cap \Lambda(x)$ for any  $f\in L_1(\mathbb{R})$ and $x\in\mathcal{X}$.
    \item $fx=0$ if $(\mathrm{supp}\,\widehat{f})\cap \Lambda(x)=\emptyset$ for  $f\in L_1(\mathbb{R})$ and $x\in\mathcal{X}$.
    \item $fx=x$ if $\Lambda(x)$ is compact and $\widehat{f}=1$ in some neighborhood of  $\Lambda(x)$.
\end{enumerate}
\end{lemma}

We remark that the last two properties of the above lemma may be strengthened for the case of an isometric representation $\mathcal{T}$ (as stated in the lemma, they are valid even for a large class of unbounded representations).

\begin{lemma}\label{basklh2.1'new}
\cite[Lemma 3.7.32]{B04}.
Let $(\mathcal{X}, \mathcal{T})$ be a Banach $L_1(\mathbb{R})$-module and the representation   $ \mathcal{T}$ be isometric. Then for $f\in L_1(\mathbb{R})$ and $x\in\mathcal{X}$ the following properties hold.
\begin{enumerate}
    \item[$5^\prime$.]  $fx=0$ if $(\mathrm{supp}\,\widehat{f})\cap \Lambda(x)$ is countable and $\widehat{f}(\lambda)= 0$ for all $\lambda\in (\mathrm{supp}\,\widehat{f})\cap \Lambda(x)$.
    \item[$6^\prime$.] $fx=x$ if $\Lambda(x)$ is a compact set with a countable boundary and $\widehat{f}=1$ on $\Lambda(x)$.
\end{enumerate}
\end{lemma}

The Beurling spectrum should be thought of as a generalization of the notion of the support as we illustrate in the following example. 

\begin{example}\label{exbsp}
For the representation $S$  on $\mathcal{X}=L_2(\mathbb{R, C})$ from Example \ref{exrepSM} the spectrum  
 $\Lambda(x, S)$ coincides with the support $\supp \widehat{x}$ of the Fourier transform of the function  $x\in\mathcal{X}$.
For the representation  $M$ of the same example   $\Lambda(x, M)= \mathrm{supp}\,x$ is the support of the function  $x\in\mathcal{X}$ itself.
\end{example}

We shall make use of the notions of a bounded approximate identity (b.a.i.) of the algebra $L_1(\mathbb R)$ and a spectral submodule.

\begin{definition}[\cite{B04, BK05}]\label{oae}
A bounded net  $(e_\alpha)_{\alpha\in\Omega}$ of functions from $L_1(\mathbb R)$ is a bounded approximate identity (b.a.i.) of the algebra $L_1(\mathbb R)$, if
\begin{enumerate}
\item $\widehat e_\alpha(0) = 1$ for all $\alpha\in \Omega$;
\item  $\lim_\alpha  e_\alpha * f = f$ for any $f\in L_1(\mathbb R)$.
\end{enumerate}
\end{definition}

\begin{definition}\label{specmod}
For $\sigma\subset\mathbb R$, a spectral submodule $\mathcal{X}(\sigma)$ consists of all $x\in\mathcal{X}$ for which $\Lambda(x)\subseteq\sigma$.
\end{definition}

Observe that a spectral submodule $\mathcal{X}(\sigma)$ is closed provided that the set $\sigma$ is closed. Consider also the two submodules
$$
\mathcal{X}_{comp}=\{x\in\mathcal{X}:\ \Lambda(x)\;\mbox{is compact}\},
$$
and
$$
\mathcal{X}_\Phi=\{fx:\;f\in L_1, x\in\mathcal{X}\},
$$
which prominently feature in the following extremely useful result.

\begin{theorem}[\cite{BK05, BKU20}]\label{cht} 
We have
$$
\mathcal{X}_c=\mathcal{X}_\Phi=
\overline{\mathcal{X}_{comp}} = \{x\in \mathcal{X}: \lim e_\alpha x = x \mbox{ for any b.a.i. } (e_\alpha) \mbox{ from } L_1(\mathbb R)\}.
$$
\end{theorem}
The above result is essentially the Cohen-Hewitt Factorization Theorem \cite{C59, H64}. Stated in the above form, it allows one to easily prove various density theorems in approximation theory as well as more difficult approximation results.

Another interesting class of submodules is formed by almost periodic vectors.

\begin{definition}[\cite{BK10}]\label{ap}
A vector $x$ from a Banach $L_1$-module $(\mathcal{X,T})$ is called almost periodic  (a.p.),
if the function $\tau_x: \mathbb{R}\to\mathcal{X}$ from Definition \ref{baskdef2.3} is continuous and almost periodic; in other words for any $\varepsilon >0$, the set $\Omega(\varepsilon) = \{t \in\mathbb R: \|\mathcal T(t)x-x \|<\varepsilon\}$ is relatively dense in $\mathbb R$, i.~e., there exists a compact set $K=K_\varepsilon\subset\mathbb R$ such that $(t+K)\cap\Omega(\varepsilon)\neq\emptyset$ for any $t\in\mathbb R$.
\end{definition}

\begin{definition}
A non-zero vector $x\in\mathcal{X}$ is called an eigenvector of the module  $\mathcal X$ corresponding to an eigenvalue $\lambda\in\mathbb R$, if $\mathcal T(t)x = e^{i\lambda t}x$, $t\in\mathbb R$.
\end{definition}

Clearly, for an eigenvector $x\in\mathcal{X}$, one has $fx=\widehat f(\lambda)x$ for any $f\in L_1$. It is known (see, e.~g., \cite{B04}) that the set of all a.p.~vectors coincides with the smallest closed submodule of $\mathcal X$ that contains all of its eigenvectors. We shall denote this submodule by $\mathcal{AP\, X}$. Additional information about a.p.~vectors can be found, for example, in \cite{BK10, B04}. We remark, in particular, that if the Beurling spectrum $\Lambda(x)$ has no limit points and $x\in\mathcal X_c$ then 
$x\in\mathcal{AP\, X}$.

\begin{definition}[\cite{BK10, B04}]\label{app}
A vector $x$ from a Banach $L_1$-module $(\mathcal{X,T})$ is called periodic if $x\in\mathcal X_c$ and there exists $\omega > 0$ such that $\mathcal T(\omega)x = x$, i.e., the function $\tau_x: \mathbb{R}\to\mathcal{X}$ from Definition \ref{baskdef2.3} is continuous and periodic. The set of all periodic vectors from $\mathcal X$ will be denoted by  $\mathcal{P(X)}$, and the set of all periodic vectors of period $\omega$ -- by $\mathcal{P_\omega(X)}$.
\end{definition}

Clearly, for any $\omega > 0$, the set $\mathcal{P_\omega(X)}$ is a closed submodule of $\mathcal{AP\, X}$. Observe also that for $x\in\mathcal{P_\omega(X)}$ one has $\Lambda(x) \subseteq \left\{\frac{2\pi n}{\omega}: n\in\mathbb Z\right\}$.

The following definition includes more exotic classes of Banach modules that use (almost) periodicity in their definition.

\begin{definition}\label{inap}
Let $(\mathcal{X,T})$ be a Banach module and $\mathcal F$ be a closed submodule of $\mathcal{X}$. The quotient space $\mathcal{X/F}$ is naturally equipped with a module structure associated with the quotient representation  $[\mathcal T/\mathcal F]:\mathbb R\to \mathcal{X/F}$,
\[
[\mathcal {T/F}](t)(x+\mathcal F) = \mathcal T(t)x+\mathcal F, \quad t\in\mathbb R,\ x\in\mathcal X.
\]
A vector $x\in\mathcal{X}$ is called (almost) periodic with respect to a submodule $\mathcal F$ (or $\mathcal F$-(almost-)periodic), if the vector $x+\mathcal F\in (\mathcal{X/F}, \mathcal{T/F})$ is (almost) periodic.
\end{definition}

The space of $\mathcal F$-almost-periodic vectors from $\mathcal X$ will be denoted by
$\mathcal{AP^F X}$ and the space of  $\mathcal F$-periodic vectors of period  $\omega$ -- by $\mathcal{P^F_\omega(X)}$.
We note that periodic and almost periodic at infinity functions that were considered, e.~g., in \cite{BSS19, BST18, B20} are examples of $\mathcal X_c$-(almost-)periodic functions in the module
$(\mathcal{X, T})$, where $\mathcal{X} = L_\infty(\mathbb{R}, \mathscr{X})$ and $\mathcal T = M$ as in Example \ref{exrepSM}.

In the following definition we introduce one of the key notions of the spectral theory of Banach $L_1(\mathbb R)$-modules.

\begin{definition}\label{baskdef2.8}
A closed linear operator $A:D(A)\subset\mathcal{X}\to\mathcal{X}$ is the generator of a Banach $L_1$-module $(\mathcal{X, T})$
if its resolvent $R: \rho(A)\to\mathrm{End}\,\mathcal{X}$ satisfies
$R(z)=\mathcal T(f_z)$ for all $z\in\mathbb{C}\setminus\mathbb{R}$,
where the functions $f_z\in L_1$ are defined via the Fourier transform: $\widehat{f}_z(\lambda)=(\lambda-z)^{-1}$, $\lambda\in\mathbb{R}$.
\end{definition}

A simple computation shows that
\begin{equation}\label{fzf}
f_z(t) = -ie^{izt}\chi_{(-\infty,0]}(t),\ z\in \mathbb {C},\ \Im z < 0.
\end{equation}
Therefore, for any $x\in\mathcal X_c$, one has
\begin{equation}\label{resf}
R(z)x = (A-zI)^{-1}x = -i\int_{-\infty}^0 e^{izt}\mathcal T(-t)x dt,\ z\in \mathbb {C},\ \Im z < 0.
\end{equation}
We also note that we always have $D(A) \subseteq \mathcal{X}_c$ due to Theorem \ref{cht}.

\begin{example}\label{exrepU}
Consider $\mathcal{X}=L_2[0, 2\pi]$, i.e., the subspace of $L_2$-functions with the support in $[0, 2\pi]$. Let $A=-d^2/dt^2: D(A)\subset\mathcal{X}\to\mathcal{X}$,
where $D(A)$ is the subset of the Sobolev space $W_2^2[0, 2\pi]$, defined by the boundary conditions $x(0)=x(2\pi)$, $x\in W_2^2[0, 2\pi]$. The eigenvalues of the operator $A$ are given by $\lambda_n=n^2$,
$n\in\mathbb{Z}$. When $n\ne 0$ the eigenvalues are semisimple and have multiplicity $2$, and the eigenvalue $\lambda_0 = 0$ is simple. The corresponding eigenvectors $e_n(t)=e^{int}/\sqrt{2\pi}$, $n\in\mathbb{Z}$, from an orthonormal basis of the Hilbert space $L_2[0, 2\pi]$. The eigenspaces $E_n$, $n\in \mathbb{Z}_+$,
corresponding to the eigenvalues $\lambda_n$, $n\in \mathbb{Z}_+$, are images of the orthogonal projections
$P_n$, $n\in \mathbb{Z}_+$, where $P_nx=\langle x, e_n\rangle e_n+\langle x, e_{-n}\rangle e_{-n}$, $n\in\mathbb{N}$, and $P_0x=\langle x, e_0\rangle e_0$. The disjunctive family of projections $\mathscr P = \{P_n, n\in \mathbb{Z}_+\}$ forms a resolution of the identity in $L_2[0, 2\pi]$ and the corresponding eigenspaces $E_n=\mathrm{Im}\,P_n$, $n\in \mathbb{Z}_+$ form an orthogonal basis of subsapces \cite{GK69} in $L_2[0, 2\pi]$.

The operator $iA$ generates a strongly continuous operator group, and the operator $A$ is the generator of the Banach $L_1(\mathbb{R})$-module with the structure that is associated with the representation $T_A: \mathbb{R}\to\mathrm{End}\,\mathcal{X}$,
\begin{equation}\label{bask2.3}
T_A(t)x=e^{itA}x=\sum_{n\in\mathbb{Z}_+}e^{i\lambda_nt}P_nx,\ t\in\mathbb R, \ x\in\mathcal{X}.
\end{equation}
Alternatively, one may introduce the representation  $T_{\mathscr P}: \mathbb{R}\to\mathrm{End}\,\mathcal{X}$, where
\begin{equation}\label{bask2.3P}
T_{\mathscr P}(t)x=\sum_{n\in\mathbb{Z}_+}e^{int}P_nx,\ t\in\mathbb R, \  x\in\mathcal{X}.
\end{equation}
One then has
$$
\Lambda(x, T_A)=\{\lambda_n\in\sigma(A): P_nx\ne 0\} \mbox{ and } \Lambda(x, T_{\mathscr P})=\{n\in\mathbb{Z}_+: P_nx\ne 0\}.
$$
A similar construction is valid for any operator with a discrete spectrum and a system of spectral projections that form a resolution of the identity
(see, e.~g., \cite{B08mz, BDS11, DM10}).
\end{example}

Under the universal assumptions of this paper, all modules have a unique well-defined generator. Moreover, if the representation $\mathcal{T}: \mathbb{R}\to\mathrm{End}\,\mathcal{X}$ is strongly continuous, then the operator $iA$ is the (infinitesimal) generator of the $C_0$-group $\mathcal{T}$ (see \cite{BKU20, EN00}). 
Hereinafter, $A: D(A)\subset\mathcal{X}\to\mathcal{X}$ will be the generator of a 
$L_1(\mathbb{R})$-module $(\mathcal{X, T})$.

The following important property of the resolvent of a module generator is very useful for us.

\begin{lemma}\label{resgen}
Let $A$ be the generator of a Banach module $(\mathcal {X, T})$, where the representation  $\mathcal T$ is bounded. Then
\[
\lim_{\varepsilon\to \pm\infty} \|(A-(\alpha+i\varepsilon)I)^{-1}\| = 0, \ \alpha\in\mathbb R.
\]
\end{lemma}

\begin{proof}
The desired property is an immediate consequence of
\[
(A-(\alpha+i\varepsilon)I)^{-1} = \mathcal T(f_z),\ z = \alpha+i\varepsilon,
\]
where the functions $f_z\in L_1$ from Definition  \ref{baskdef2.8} satisfy
\[
\lim_{\Im z\to-\infty} \|f_z\| = 0,
\]
due to \eqref{fzf}.
\end{proof}

The following lemma provides an equivalent definition for the generator of a module. The statement uses the family of functions $\chi_t \in L_1$ given by $\chi_t = i\chi_{[-t,0]}$, $t > 0$.

\begin{lemma}\label{equigen}
An operator $A$ is the generator of a Banach module $(\mathcal {X, T})$ if and only of its domain consists of all $x\in \mathcal X$ for which there exists
 $y \in\mathcal X$ satisfying $\mathcal T(t)x - x =  \mathcal T(\chi_t)y$ { for all } $t> 0$, in which case $y = Ax$.
\end{lemma}

\begin{proof}
We observe that, given $x\in \mathcal X$, there is at most one
 $y \in\mathcal X$ satisfying $\mathcal T(t)x - x =  \mathcal T(\chi_t)y$  for all  $t> 0$. Indeed, from non-degeneracy of the module $(\mathcal X,\mathcal T)$, we see that if $\mathcal T(\chi_t)v = 0$  for all  $t> 0$ then $\Lambda(v) = \emptyset$ implying $v = 0$. Thus, the second part of the statement of the lemma gives a well-defined linear operator. We need to prove that it coincides with the generator of the module. 

Assume that $A$ is the generator of a Banach module $(\mathcal {X, T})$ and
$x = \mathcal T(f_z)u = (A-zI)^{-1}u$
for some $z\in \mathbb {C\setminus R}$ and $u\in \mathcal X$. Then, for $y=Ax= u+zx$, we get
\[
\mathcal T(t)x - x -  \mathcal T(\chi_t)y = \mathcal T(S(t)f_z-f_z-\chi_t+z(f_z*\chi_t))u  = 0,
\]
due to that fact that, for $g = S(t)f_z-f_z-\chi_t+z(f_z*\chi_t)\in L_1$, one has
\[
\widehat g(\lambda) = \frac{e^{i\lambda t} -1 -\lambda \widehat{\chi}_t(\lambda)}{\lambda - z} = 0,\quad \lambda\in\mathbb R,\ t > 0, \ z\in \mathbb {C\setminus R}.
\]

Conversely, assume that for $x\in \mathcal X$ there exists $y \in\mathcal X$,  such that $\mathcal T(t)x - x =  \mathcal T(\chi_t)y$ { for all } $t> 0$. Let $u = y - zx$, $z\in \mathbb {C}$, $\Im z < 0$. Then for any $f\in L_1$, using  $\widehat f_z(0) = - z^{-1}$, we get
\[
\begin{split}
\mathcal T(f) (\mathcal T(f_z)u - x) &= \mathcal T(f* f_z)y - z \mathcal T(f*f_z)x - \mathcal T(f)x \\
& = \mathcal T(f* f_z)y - z\int f_z(s)\mathcal T(f)(\mathcal T(-s)x - x)ds \\
& = \mathcal T(f* f_z)y + zi\int_{-\infty}^0 e^{izs} \mathcal T(\chi_{-s})\mathcal T(f)yds \\
& = \mathcal T(f* f_z)y - z\int_{-\infty}^0\int_s^0  e^{izs}\mathcal T(-\tau)\mathcal T(f)yd\tau ds \\
& = \mathcal T(f* f_z)y - z\int_{-\infty}^0\int_{-\infty}^{\tau}  e^{izs}\mathcal T(-\tau)\mathcal T(f)ydsd\tau
\\& = \mathcal T(f* f_z)y + i\int_{-\infty}^0  e^{iz\tau}\mathcal T(-\tau)\mathcal T(f)yd\tau
\\& = \mathcal T(f* f_z)y -  \mathcal T(f_z* f)y= 0,
\end{split}
\]
and  the equality $x = \mathcal T(f_z)u$ follows since the module $\mathcal{(X,T)}$ is non-degenerate. Similar equalities hold for $z\in \mathbb {C}$ with $\Im z > 0$ and the lemma is proved.
\end{proof}

The following useful result   contains a version of the spectral mapping theorem. It follows from \cite[Theorem 3.3.11]{B04} and \cite[Theorem 2.12]{BK14}
(see also \cite{B79, BD19}). 

\begin{theorem}\label{smt}
Let $A$ be the generator of a Banach $L_1$-module $(\mathcal{X, T})$. Then
\[
\sigma(A) = \Lambda(\mathcal{X, T}).
\]
Let also $K$ be a compact subset of $\mathbb R$ and $A_K$ be the restriction of the generator  $A$ to the spectral submodule $\mathcal X(K)$. Then $A_K \in \mathrm{End}\,\mathcal{X}(K)$ is the generator of the Banach $L_1$-module $(\mathcal{X}(K), \mathcal{T}_K)$, where the representation $\mathcal{T}_K: \mathbb R \to \mathrm{End}\,\mathcal{X}(K)$ is given by $\mathcal T_K(t)x = e^{itA_K}x = \mathcal T(t)x$, $t\in\mathbb R$, $x\in\mathcal X(K)$. In this setting, $\sigma(A_K) = \Lambda(\mathcal{X, T})\cap K$, and the norm of the operator $A_K$ coincides with its spectral radius. Moreover, the representation $\mathcal{T}_K$ is continuous in the uniform operator topology and admits a holomorphic extension to an entire function $\mathcal{T}_K: \mathbb C \to \mathrm{End}\,\mathcal{X}(K)$, $\mathcal{T}_K(z) = e^{izA_K} = \mathcal{T}_K(t)e^{-i\alpha A_K} = e^{-i\alpha A_K} \mathcal{T}_K(t)$, $z = t+i\alpha$, for which
\begin{equation}\label{cnormest}
\|\mathcal T_K(t+i\alpha)\| =\|\mathcal T_K(i\alpha)\|\le \max_{\lambda\in\sigma(A_K)} e^{-\alpha\lambda},\ t,\alpha\in\mathbb R.
\end{equation}
\end{theorem}

\begin{corollary}\label{cnormest1}
Let $x\in (\mathcal{X, T})$ satisfy $\Lambda(x, \mathcal T) \subseteq [-a,a]$ for some $a\ge 0$. Then the function $\tau_x: \mathbb{R}\to\mathcal{X}$ from Definition  \ref{baskdef2.3} admits a holomorphic  extension to an entire function of exponential type and 
\begin{equation}\label{cnormest2}
\|\tau_x(t+i\alpha)\| =\|\tau_x(i\alpha)\|\le  e^{\alpha a},\ t,\alpha\in\mathbb R.
\end{equation}
\end{corollary}

We note that $x\in (\mathcal{X, T})_c$ satisfies $\Lambda(x, \mathcal T) \subseteq [0,\infty)$ if and only if the function  $\tau_x: \mathbb{R}\to\mathcal{X}$ from Definition  \ref{baskdef2.3} has a bounded continuous  extension to the upper half-plane of $\mathbb C$, which is holomorphic in its interior (see \cite[Lemma 8.2]{BK05}). 

The proofs of our results on similarity of linear operators require introduction of Banach module structures on large subspaces of the Banach space $\mathrm{End}\,\mathcal{X}$.

We define the representation  $\mathcal T_0: \mathbb{R}\to\mathrm{End}\,\mathrm{End}\,\mathcal{X}$, by
\begin{equation}\label{T0}
\mathcal T_0(t)B=\mathcal{T}(t)B\mathcal{T}(-t), \quad t\in\mathbb{R}, \quad B\in\mathrm{End}\,\mathcal{X}.
\end{equation}

Let $\mathcal{M}_0$ be the subspace of $\mathrm{End}\,\mathcal{X}$ that consists of all operators $B \in\mathrm{End}\,\mathcal{X}$ such that the function $t \mapsto \mathcal T_0(t)B$ is continuous in the strong operator topology and the nets $\mathcal T(f)(B\mathcal T(e_\alpha) - \mathcal T(e_\alpha)B)$ converge to $0$ (in the strong operator topology) for any function $f\in L_1$ and some b.a.i.~$(e_\alpha)$ from $L_1$. Note that if the representation  $\mathcal T$ is strongly continuous then $\mathcal{M}_0 = \mathrm{End}\,\mathcal{X}$. Crucially for us, the formula  $$
(fB)x=(\mathcal T_0(f)B)x=\int_{\mathbb{R}}f(t)(\mathcal T_0(-t)B)x\,dt, \quad f\in L_1, \quad x\in\mathcal{X},
$$
defines a structure of a non-degenerate Banach $L_1(\mathbb{R})$-module on $\mathcal{M}_0$, that is associated with the representation  $\mathcal T_0$ (see \cite[Lemma 5.11]{BK05}). It is also worth noting that operators in $\mathcal M_0$ are completely determined by their values on $\mathcal{X}_{comp}$.

\begin{lemma}\label{basknewlh2.2}
Assume that operators $F$ and $G$ belong to $\mathcal{M}_0$ and $x\in\mathcal{X}$. The following properties hold. 
\begin{enumerate}
    \item $\Lambda(Fx, \mathcal{T})\subseteq \overline{\Lambda(F, \mathcal T_0)+\Lambda(x, \mathcal{T})}$.
    \item $\Lambda(FG, \mathcal T_0)\subseteq \overline{\Lambda(F, \mathcal T_0)+\Lambda(G, \mathcal T_0)}$.
    \item $\Lambda(F, \mathcal T_0)\subseteq \overline{\Lambda(\mathcal{X}, \mathcal{T})-\Lambda(\mathcal{X}, \mathcal{T})}=
\overline{\sigma(A)-\sigma(A)}$.
\end{enumerate}
\end{lemma}

\begin{proof}
Properties 1 and 2 are proved in \cite[Corollaries 5.15 and 7.8]{BK05}. The equality in 3 follows from Theorem \ref{smt}. Let us prove the containment in 3 using 1.
Let
$$
\Delta=\overline{\Lambda(\mathcal{X}, \mathcal{T})-\Lambda(\mathcal{X}, \mathcal{T})},
$$
and $\lambda\notin\Delta$. Consider $f\in L_1(\mathbb{R})$ such that $\widehat{f}(\lambda)\ne 0$ and
$\mathrm{supp}\,\widehat{f}\cap\Delta=\emptyset$. Then, according to property $1$, given any $x\in\mathcal{X}$, we have
\begin{flalign*}
\Lambda((\mathcal T_0(f)F)x, \mathcal{T}) &\subseteq \overline{\Lambda(\mathcal T_0(f)F, \mathcal T_0)+\Lambda(x, \mathcal{T})}\subseteq
\overline{\mathrm{supp}\,\widehat{f}+\Lambda(\mathcal{X}, \mathcal{T})}\subseteq\\
&\subseteq \Delta^c+\Lambda(\mathcal{X}, \mathcal{T})
\subseteq (\Lambda(\mathcal{X}, \mathcal{T}))^c.
\end{flalign*}
Consequently, $\Lambda((\mathcal T_0(f)F)x, \mathcal{T})=\emptyset$ and $(\mathcal T_0(f), F)x=0$ by property 4 of Lemma \ref{basklh2.1new}.
Since $x\in\mathcal{X}$ was chosen arbitrarily, we get $\mathcal T_0(f)F=0$, and, therefore, $\lambda\notin\Lambda(F, \mathcal T_0)$.
\end{proof}

We also provide the following useful result about the generator of the module $(\mathcal{M}_0, \mathcal{T}_0)$.

\begin{theorem}
Let $A$ be the generator of a Banach $L_1$-module $(\mathcal{X, T})$. Then the generator of the Banach module $(\mathcal{M}_0, \mathcal{T}_0)$ is the commutator $[A]: D([A])\subseteq \mathcal{M}_0 \to \mathcal{M}_0$, $[A]X = AX - XA$, $X\in D([A])$, where 
$$
D([A]) = \{X \in(\mathcal{M}_0)_c: AX - XA \mbox{ extends uniquely to an operator in } \mathcal M_0\}.
$$
\end{theorem}

\begin{proof}
Let $A_0$ be the generator of the Banach module $(\mathcal{M}_0, \mathcal{T}_0)$.
As we mentioned above, $D(A_0) \subseteq (\mathcal{\mathcal{M}}_0)_c$ due to Theorem \ref{cht}. Moreover, if $\mathcal M_0 = (\mathcal{\mathcal{M}}_0)_c$, it is well known \cite{BKU20, EN00} that $A_0 = [A]$. Hence, in general, $A_0$ coincides with $[A]$ on the domain $\mathcal D_0$ of the generator of the module $(\mathcal{M}_0, \mathcal{T}_0)_c$.
Let us show that the two operators are, indeed, the same. Since $A$ commutes with each $T(t)$, $t\in\mathbb R$, we have
\[
  (\mathcal{T}_0(f)[A]X)x  = \int f(t)\mathcal T(-t)(AX-XA)\mathcal T(t)x dt 
  = \int f(t)(A\mathcal T(-t)X\mathcal T(t)-\mathcal T(-t)X\mathcal T(t)A)x dt 
\]
for each $f\in L_1$,  $X\in D([A])$, and $x \in \mathcal{X}_{comp}$ such that $XT(t)x \in D(A)$ for all $t\in\mathbb R$. It follows that $[A]\mathcal{T}_0(f) = \mathcal{T}_0(f)[A]$ on $D([A])$ for every $f\in L_1$. Therefore, for each $X\in D([A])\cap D(A_0)$ and $f\in L_1$ with $\supp \widehat f$ compact we have 
\[
\mathcal{T}_0(f)([A]X) = [A](\mathcal{T}_0(f)X) = A_0(\mathcal{T}_0(f)X) = \mathcal{T}_0(f)(A_0X),
\]
where the second equality holds because $\mathcal{T}_0(f)X$ is in $\mathcal{D}_0$, and the last equality follows from \cite[(2.2) and (2.5)]{BKU20}.
Now, non-degeneracy of the module $\mathcal M_0$ implies that $A_0$ coincides with $[A]$ on $D([A])\cap D(A_0)$.
It remains to prove that $D([A])= D(A_0)$.

Assume that $X \in D([A])$ and $Y = [A]X$. Then for any $t > 0$, $\chi_t$ from Lemma \ref{equigen}, and $f\in L_1$ with $\supp \widehat f$ compact we have 
\[
\begin{split}
    \mathcal{T}_0(f)\left(\mathcal{T}_0(t)X - X - \mathcal{T}_0(\chi_t)Y\right) & = \mathcal{T}_0(t)\mathcal{T}_0(f)X - \mathcal{T}_0(f)X - \mathcal{T}_0(\chi_t)[A](\mathcal{T}_0(f)X) \\& =  \mathcal{T}_0(t)\mathcal{T}_0(f)X - \mathcal{T}_0(f)X - \mathcal{T}_0(\chi_t)A_0(\mathcal{T}_0(f)X) = 0
\end{split}
\]
by Lemma \ref{equigen} since $ (\mathcal{T}_0(f)X) \in \mathcal D_0 \subset D(A_0)\cap D([A])$. Consequently, $X \in D(A_0)$ is implied by non-degeneracy of $\mathcal M_0$. Thus, $D([A]) \subseteq D(A_0)$. 

Assume $X \in D(A_0)$. Then there is $Y\in \mathcal{M}_0$ such that $X = \mathcal T_0(f_z)Y$, where $f_z$ is given by \eqref{fzf} for some $z\in \mathbb {C}$ with $\Im z < 0$. To finish the proof of the theorem it suffices to show that for any $x \in \mathcal X_{comp}$ we have $(AX-XA)x = (Y+zX)x.$ We remark that for such $x$ one has $Xx\in D(A)$ since $X\in (\mathcal M_0)_c$ and the  operator $A$ is closed. Let $f\in L_1$ be an arbitrary function with $\supp{\widehat{f}}$ compact and $y\in \mathcal X_{comp}$ be such that $x = \mathcal T(f_z)y$; such $y$ exists because $\mathcal X_{comp} \subseteq D(A)$. Due to non-degeneracy of the module $\mathcal{X}$ we only need to show that
\[
\mathcal T(f)(AX-XA)\mathcal T(f_z)y =\mathcal T(f)(Y+zX)\mathcal T(f_z)y.
\]
However, since $[A]$ and $A_0$ agree on $\mathcal T_0(f_z)(\mathcal{M}_0)_c$ and $\mathcal T(f)X\mathcal T(f_z)\in (\mathcal{M}_0)_c$, we do, indeed, have
\[
\begin{split}
    \mathcal T(f)(AX-XA)\mathcal T(f_z)y & = (A\mathcal T(f)X\mathcal T(f_z)-\mathcal T(f)X\mathcal T(f_z)A)y \\
    &= (A\mathcal T(f)(\mathcal T_0(f_z)Y)\mathcal T(f_z)-\mathcal T(f)(\mathcal T_0(f_z)Y)\mathcal T(f_z)A)y \\ &= (A\mathcal T_0(f_z)(\mathcal T(f)Y\mathcal T(f_z))-\mathcal T_0(f_z)(\mathcal T(f)Y\mathcal T(f_z))A)y \\
    &= (\mathcal T(f)Y\mathcal T(f_z)+z\mathcal T_0(f_z)(\mathcal T(f)Y\mathcal T(f_z)))y = \mathcal T(f)(Y+zX)\mathcal T(f_z)y.
\end{split}
\]
Thus, $AX-XA$ extends to the operator $Y+zX \in \mathcal M_0$. Uniqueness of the extension follows since operators in $\mathcal{M}_0$ are completely determined by their values on $\mathcal X_{comp}$.
\end{proof}

The following lemma lies at the heart of our construction of admissible triples in the method of similar operators.

\begin{lemma}[\cite{BKU20}]\label{basklh2.1}
Assume that functions $\varphi, \psi\in L_1\cap L_2$ are such that $\widehat{\varphi}=1$ in some neighborhood of $0$ and
\begin{equation}\label{bask2.1}
\widehat{\psi}(\lambda)=(1-\widehat{\varphi}(\lambda))/\lambda, \quad \lambda\in\mathbb{R}\setminus\{0\},
\quad \widehat{\psi}(0)=0.
\end{equation}
Then for operators $\mathcal T_0(\varphi)X$ and $\mathcal T_0(\psi)X$, $X\in\mathcal{M}_0$, one has
$$
A(\mathcal T_0(\psi)X)x-(\mathcal T_0(\psi)X)Ax=Xx-(\mathcal T_0(\varphi)X)x, \quad x\in D(A),
$$
or, equivalently,
$$
[A](\mathcal T_0(\psi)X) = X-\mathcal T_0(\varphi)X\in\mathcal M_0.
$$
\end{lemma}

In our constructions, we will use specific functions in place of $\varphi$ and $\psi$. They are defined as follows.

For $a>0$, we let $\tau_a$ be the ``trapezoid'' function
$$
\tau_a(\lambda)=
\begin{cases}
1, \quad &|\lambda|\leqslant a, \\
\frac{1}{a}(2a-|\lambda|), \quad &a<|\lambda|\leqslant 2a,\\
0, \quad &|\lambda|>2a.
\end{cases}
$$
One easily verifies that $\tau_a\in L_2(\mathbb{R})$ and $\|\tau_a\|_2\leqslant 2\sqrt{2a/3}$,
$\|\tau_a'\|_2=\sqrt{2/a}$. Moreover, $\tau_a=\widehat{\varphi}_a$, where
\begin{equation}\label{otrap}
\varphi_a(t)=\frac{2\sin\frac{3at}{2}\sin\frac{at}{2}}{\pi at^2}.
\end{equation}
Functions $\tau_a$ have been widely used for other purposes, see, e.g., \cite{Le53}. Note that the net $(\varphi_a)$, $a\in\mathbb R_+$, $a\to \infty$, forms a b.a.i.~in $L_1(\mathbb{R})$.

Furthermore, consider the functions $\omega_a$, $a>0$, defined by
$$
\omega_a(\lambda)=(1-\tau_a(\lambda))/\lambda=
\begin{cases}
0, \quad &|\lambda|\leqslant a, \\
-\frac{1}{a}-\frac{1}{\lambda}, \quad &-2a\leqslant \lambda< -a,\\
\frac{1}{a}-\frac{1}{\lambda}, \quad &a<\lambda\leqslant 2a,\\
\frac{1}{\lambda}, \quad &|\lambda|>2a.
\end{cases}
$$
It is easy to see that $\|\omega_a\|_2=\sqrt{(4-4\ln 2)/a}\leqslant (1.11)/\sqrt{a}$ and
$\|\omega_a'\|_2=\sqrt{2/(3a^3)}\leqslant (0.82)/(a\sqrt{a})$. We then define the functions $\psi_a$ by letting 
\begin{equation}\label{psia}
    \widehat{\psi}_a=\omega_a,\quad a>0.
\end{equation}

The following lemma, which was proved for example in \cite{BKU20}, shows that the functions $\psi_a$ indeed belong to $L_1$ and allows one to obtain estimates for
  $\|\varphi_a\|_1$ and $\|\psi_a\|_1$.  
  
\begin{lemma}\label{basklh2.2}
Let $f\in L_2$ be absolutely continuous and such that $f'\in L_2$. Then one has $f\in L_1$ and $\|f\|_1^2\leqslant 2\|\widehat{f}\|_2\|\widehat{f}'\|_2$.
\end{lemma}

 Lemma \ref{basklh2.2} yields the estimates $\|\varphi_a\|_1\leqslant 2^{\frac{3}{2}}3^{-\frac{1}{4}}$ and
\begin{equation}\label{psiest}
\|\psi_a\|_1\leqslant \frac{2}{a}\sqrt[4]{2(1-\ln 2)/3}\leqslant 1.35/a.    
\end{equation}
Note that for functions $\varphi_a$
more precise estimates can be obtained: $\|\varphi_a\|_1\leqslant 4/\pi+(2\ln 3)/\pi$ (see~\cite{Le53}) and
$\|\varphi_a\|_1\leqslant \sqrt{3}$ (see~\cite{RS00}).

We conclude this section with a lemma which, in some way, complements Corollary \ref{cnormest1} and will be used in the proof of one of the main results in Section \ref{basksec4}. The lemma can be deduced from \cite[Corollary 1]{BS09} but we will provide a proof for completeness.

\begin{lemma}\label{hcest}
    Assume that for some $b\ge a>0$ the function $\psi_a$ is given by \eqref{psia} and a vector $x\in (\mathcal{X},\mathcal{T})$ satisfies either $\Lambda(x) \subseteq [2b,\infty)$  or  $\Lambda(x) \subseteq (-\infty, -2b]$. Then $\mathcal T(\psi_a)x = \mathcal T(\psi_b)x$  
    and $\|\mathcal T(\psi_a)x\| \le \frac{\|x\|}{2b}$.
\end{lemma}

\begin{proof}
    The equality $\mathcal T(\psi_a)x - \mathcal T(\psi_b)x = \mathcal T(\psi_a-\psi_b)x = 0$ follows immediately from Lemma \ref{basklh2.1'new}($5^\prime$) and we only need to obtain the estimate for the norm. If we use \eqref{psiest} with $a= b$, we will only get the estimate $\|\mathcal T(\psi_a)x\| \le \frac{1.35\|x\|}{b}$. Thus, we need a more elaborate argument, which we borrow from the proof of \cite[Theorem 1]{BS09}. We consider a family of functions $\widetilde \psi_b$ defined via their Fourier transform by
    \begin{equation}\label{ptf}
       \widehat{\widetilde \psi_b}(\lambda) = (|\lambda-2b|+2b)^{-1}, \quad b>0. 
    \end{equation}
    Using Lemma \ref{basklh2.1'new}($5^\prime$) once again, we see that $\mathcal T(\psi_a)x = \mathcal T(\widetilde\psi_b)x$.
    Thus, the postulated estimate would follow if we can show that $\|\widetilde\psi_b\|_1 = \frac1{2b}$. The latter assertion, however, is an immediate consequence of P\'olya's criterion for positive definiteness of functions (see \cite[Theorem 4.3.1]{L70}). Indeed, the function $g = 2b\widehat{\widetilde \psi_b}(\cdot+2b)$ is even, convex on $\mathbb R_+$, and satisfies $g(0) = 1$ together with $\lim\limits_{\lambda\to\infty}g(\lambda) = 0$. Thus, by P\'olya's criterion, $g$ is the Fourier transform of a non-negative function yielding $\|\widetilde\psi_b\|_1 = \frac1{2b}g(0) = \frac{1}{2b}$ as desired.
\end{proof}

\section{$A$-harmonious spaces}\label{HarmSp}

In this section, we define the notion of submodules that are harmonious to the module generator $A$. These modules (under a few extra assumptions) will be used as the spaces of admissible perturbations when we construct admissible triples of the method similar operators.

\begin{definition}
Let $A: D(A)\subset\mathcal{X}\to\mathcal{X}$ be a closed linear operator. A linear operator 
$X: D(X)\subset\mathcal{X}\to\mathcal{X}$ is called (relatively) $A$-bounded if
$D(A)\subseteq D(X)$ and $\|X\|_A=\inf\{C>0: \|Xx\|\leqslant C(\|x\|+\|Ax\|), x\in D(A)\}<\infty$.
\end{definition}

The space $\mathfrak{L}_A(\mathcal{X})$ of all $A$-bounded linear operators (with the domain $D(A)$)
is a Banach space with the norm $\|X\|_A$. Moreover,
$\mathrm{End}\,\mathcal{X}\subset\mathfrak{L}_A(\mathcal{X})$,
and $\mathfrak{L}_A(\mathcal{X})$ has equivalent norms defined by $\|X\|_\lambda=\|X(A-\lambda I)^{-1}\|$,
$\lambda\in\rho(A)$.

\begin{definition}\label{baskdef16}
Let $A: D(A)\subset\mathcal{X}\to\mathcal{X}$ be the generator of a Banach
$L_1(\mathbb{R})$-module $(\mathcal{X}, \mathcal{T})$ and $\mathcal{M}$ be a linear subspace of $\mathfrak{L}_A(\mathcal{X})$. A space $\mathcal{M}$ is called 
{{harmonious to  $A$}} (or $A$-harmonious), if the following properties hold.
\begin{enumerate}
    \item $\mathcal{M}$ is continuously and injectively embedded in $\mathfrak{L}_A(\mathcal{X})$.
    \item Formula \eqref{T0} is valid for any  $B\in\mathcal{M}$ and defines a bounded representation of the group $\mathbb{R}$ by operators from $\mathrm{End}\, \mathcal{M}$; we retain te symbol $\mathcal T_0$  for this representation, which can be considered isometric in view of \ref{ogriz}.
    \item The space $\mathcal{M}$ is a non-degenerate Banach $L_1(\mathbb{R})$-module with the structure associated with the representation  $\mathcal T_0$.
\end{enumerate}
\end{definition}

The space $\mathcal{M}_0$ introduced ahead of Lemma \ref{basknewlh2.2} is the main example of an $A$-harmonious space.  In fact, all harmonious spaces we consider are continuously and injectively embedded into   $\mathcal{M}_0$. We will use the following spaces and operators to introduce other examples of harmonious spaces.

\begin{example}\label{exa4}
We shall denote by $V$ the multiplication operator $(Vx)(t) = v(t)x(t)$, $v\in L_\infty(\mathbb R)$, $t\in\mathbb R$, in the Banach module $(\mathcal{X, T}) = (L_2, S)$ of Example \ref{exbsp}. In the same module, we will also consider the operators
 $D_a$, $a\in\mathbb R\setminus \{0\}$,  defined by $(D_a x)(t) = x(at)$, $t\in\mathbb R$, $x\in L_2$. For various $a\in\mathbb R\setminus \{0\}$ these operators are contractions, dilations, or involutions.
Formula \eqref{T0} yields the following relations for these operators:
\begin{equation}\label{T0mult}
((\mathcal T_0(s)V)x)(t) = v(s+t)x(t), \ s,t\in \mathbb R, x\in\mathcal X;
\end{equation}
\begin{equation}\label{T0dial}
((\mathcal T_0(s)D_a)x)(t) = x(at+(a-1)s), \ s,t\in \mathbb R, x\in\mathcal X;
\end{equation}
\begin{equation}\label{T0inv}
((\mathcal T_0(s)VD_{-1})x)(t) = v(s+t)x(-t-2s), \ s,t\in \mathbb R, x\in\mathcal X.
\end{equation}
Below, for example in Section \ref{basksec5}, we may work with a more general setting where the multiplier function $v$ is operator-valued: $v\in L_\infty(\mathbb R, \mathrm{End}\, \mathcal X)$. The above formulas  remain valid in that setting as well.
\end{example}

\begin{example}\label{exa5}
Let the Banach space $\mathcal X = \ell_2(\mathbb Z)$, be equipped with a Banach  $L_1(\mathbb R)$-module structure that is associated with the representation   $\mathcal T = M$ define analogously to \eqref{basknew1'}:
\begin{equation}\label{modl}
M(t)x(n) = e^{int}x(n), \ x\in \ell_2(\mathbb Z),\ n\in\mathbb Z, \ t\in \mathbb R.
\end{equation}
Operators from $\mathrm{End}\, \ell_2(\mathbb Z)$ are conveniently defined by their matrices in the standard orthonormal basis  $\{e_n,\ n\in\mathbb Z\}$ of $\ell_2(\mathbb Z)$. From Definition \ref{baskdef2.8}, it follows that the generator of the module $(\mathcal X, M)$ is defined by the diagonal matrix
\[
A \sim {\mathrm{diag}}(\mathbb Z) = \left(
\begin{array}{ccccc}
 \ddots & \cdots &\vdots &\cdots & \\
 \cdots & -1 &0 &0 & \cdots  \\
 \cdots& 0 &0 & 0 & \cdots \\
 \cdots& 0 &0 & 1 & \cdots \\
 \cdots & \cdots &\vdots &\cdots & \ddots  \\
\end{array}
\right).
\]
By $L$ we shall denote operators defined by a Laurent matrix \cite{BS99}, i.~e., operators satisfying
\[
(Lx)(n) = (a*x)(n) = \sum_{m\in\mathbb Z} a(n-m)x(m), \ a\in\ell_\infty(\mathbb Z), \ x\in\ell_1(\mathbb Z)\subset \ell_2(\mathbb Z),\ n\in\mathbb Z.
\]
By $J_k$, $k\in\mathbb Z\setminus\{0\}$, we shall denote operators $(J_k x)(n) = x(kn)$, $n\in\mathbb Z$, $x\in\ell_2$. Finally, by $H$, we shall denote operators defined by Hankel matrices, i.~e., those that satisfy $H = J_{-1}L$ for some Laurent operator $L$.

Symbols $X_n$ or $(X)_n$ will refer to the operator in  $\mathrm{End}\, \mathcal X$, defined by the $n$-th diagonal of the matrix of the operator $X \in \mathrm{End}\, \mathcal X$.
With this notation, formula \eqref{T0} becomes
\begin{equation}\label{T0mod}
(\mathcal T_0(t)X)x = \sum_{n\in\mathbb Z} e^{int}X_n x, \ t\in \mathbb R, x\in\mathcal X.
\end{equation}
\end{example}

To define certain classes of $A$-harmonious spaces we shall need the notion of a weight.

\begin{definition}[\cite{BK18c}]\label{baskdef3.1}
A weight is an even measurable function  $\nu: \mathbb R \to \mathbb R$ such that $\nu(t)\ge 1$for all $t\in\mathbb R$. A bi-weight is a measurable function  $\omega: \mathbb R^2 \to \mathbb R$ such that
$\omega(s,t)=\omega(t,s)\ge 1$, $s,t\in \mathbb R$,  and $\sup_{s\in \mathbb R}\omega(s,s+t)<\infty$, $t\in \mathbb R$. For a bi-weight $\omega$, the function $\nu_\omega: \mathbb R \to \mathbb R$, defined by $\nu(t) = \sup_{s\in \mathbb R}\omega(s,s+t)$, $t\in \mathbb R$, is a weight which we call a majorante of the bi-weight $\omega$.
\end{definition}

We shall make use of the following classes of weights  (see also \cite{B90, G07col}).

\begin{definition}[\cite{BK18c}]\label{baskdef3.2}
A weight $\nu: \mathbb R \to \mathbb R$ is
\begin{itemize}
\item {submultiplicative}, if $\nu(s+t) \le \nu(s)\nu(t)$, $s, t\in\mathbb R$;
\item {subconvolutive},  if $\nu^{-1}\in L_1({\mathbb R})$ and $\nu^{-1} * \nu^{-1} \le C\nu^{-1}$ (pointwise) for some $C >0$;
\item {quasi-subconvolutive},  if
$u^{-1} = ((1+|\cdot|)\nu)^{-1}\in L^1({\mathbb R})$, and
\begin{equation}\label{qsc}
\int_{|t|\ge \alpha}\frac{\nu(s)}{u(t)\nu(s-t)}dt \to 0
\end{equation}
as $\alpha \to \infty$ uniformly in $s\in{\mathbb R}$;
\item {balanced}, if there exist $a,b\in(0,\infty)$, such that
\[
a\le \inf_{t\in[0,1]}\frac{\nu(s+t)}{\nu(s)} \le \sup_{t\in[0,1]}\frac{\nu(s+t)}{\nu(s)} \le b, \ s\in{\mathbb R};
\]
\item {subexponential}, if for any $\gamma > 1$ there exists $M > 0$ such that $\nu(t) \le M\gamma^{|t|}$ for all $t\in{\mathbb R}$;
\item a {GRS-weight} \cite{GRS64}, if it is submultiplicative and  $\lim\limits_{n\to\infty} [\nu(nt)]^{1/n} = 1$.
\end{itemize}
Each of the above notions applies to a bi-weight $\omega$ if it applies to its majorante $\nu_\omega$.
\end{definition}

\begin{example}
A typical weight is given by
\begin{equation}\label{typw}
\nu(t) = e^{a|t|^b}(1+|t|)^p.
\end{equation}
Such a weight  $\nu$ is balanced if $0\le b\le 1$ and submultiplicative if, additionally, $a, p \ge 0$; under these conditions the weight $\nu$ is a GRS-weight if and only if $0\le b<1$; the latter condition also makes the weight $\nu$ subexponential. A weight $\nu$ in \eqref{typw} is subconvolutive if
 $a=0$ and $p > 1$. It is quasi-subconvolutive if $a=0$ and $p >0$. Observe that the weight $\nu(t) = 1+|t|$   is quasi-subconvolutive but not subconvolutive.
\end{example}

The rest of this section contains examples of various  $A$-harmonious spaces and other interesting classes of linear operators.

\subsection{$A$-harmonious spectral submodules}\label{spsub}

Let $\mathcal{M}$ be an $A$-harmonious space that is continuously and injectively embedded into  $\mathcal{M}_0$.
\begin{definition}[\cite{BK05}]\label{baskdef2.9}
The set $\Lambda(X, \mathcal T_0)\setminus\{0\}$ is called the memory of the operator $X\in\mathcal{M}$. If
$\Lambda(X, \mathcal T_0)\subseteq\{0\}$, the operator $X\in\mathcal{M}$ is called memoryless. An operator $X\in\mathcal{M}$
is called causal (anticausal), provided that $\Lambda(X, T_0)\subseteq [0, +\infty)$
$(\Lambda(X, T_0)\subseteq (-\infty, 0]))$. An operator  $X\in\mathcal{M}$
is called hypercausal (hyperanticausal), if $\Lambda(X, \mathcal T_0)\subset (0, +\infty)$
($\Lambda(X, \mathcal T_0)\subset (-\infty, 0)$, respectively). 
\end{definition}

The set of all memoryless operators will be denoted by $\mathscr M$. By $\mathscr C$, $\mathscr {A\!\!C}$, $\mathscr {H\!\!C}$, and $\mathscr {H\!\!A\!\!C}$ we shall denote the sets of all causal, anticausal, hypercausal, and hyperanticausal operators, respectively.
Additionally, we let $\mathscr{H\!\!C}_a = \{X\in\mathcal M: \Lambda(X, \mathcal T_0)\subseteq [a, +\infty)\}$ and
$\mathscr{H\!\!A\!\!C}_a = \{X\in\mathcal M: \Lambda(X, \mathcal T_0)\subseteq (-\infty, -a]\}$, $a >0$.

By definition, $\mathscr M$, $\mathscr C$, $\mathscr {A\!\!C}$, $\mathscr {H\!\!C}_a$ and $\mathscr {H\!\!A\!\!C}_a$ are closed spectral submodules of $\mathcal M$ (see Definition \ref{specmod}). Clearly, any spectral submodule is an $A$-harmonious space. Additionally, due to  Lemma \ref{basknewlh2.2}(2), the above classes of operators are Banach algebras.
The following two lemmas collect a few other useful properties of the  classes. The first of them follows directly from Lemma~\ref{basklh2.1new} and Definition \ref{baskdef2.8}.

\begin{lemma}\label{comlemma}
Given an operator $X\in\mathcal{M}$, the following properties are equivalent.
\begin{enumerate}
    \item $X$ is memoryless.
    \item $\mathcal T_0(t)X=X$ for any
$t\in\mathbb{R}$.
    \item $AXx = XAx$ for any $x\in D(A)$.
\end{enumerate}
\end{lemma}

In view of the second property in the above lemma, memoryless operators are eigenvectors of the Banach  $L_1(\mathbb{R})$-module
$(\mathcal{M}, \mathcal T_0)$, see \cite[Definition~4.5]{BK05}.

\begin{lemma}\label{causidem}
Assume $\mathcal M \in \{\mathscr {H\!\!C}_{a},  \mathscr {H\!\!A\!\!C}_{a}\}$ for some $a>0$ and the operator $I+Z$, $Z\in\mathcal M$, is an idempotent. Then $Z = 0$.
\end{lemma}

\begin{proof}
Without loss of generality, we assume $\mathcal M = \mathscr {H\!\!C}_{a}$.

Assume for contradiction that $Z\neq 0$. Let $b = \inf \Lambda (Z, \mathcal T_0)$ so that $b\ge a > 0$. Let $f \in L_1$ be such that $\widehat f(b) \neq 0$ and $(\supp \widehat f) \cap [2b,\infty) =\emptyset$. Then by Lemma \ref{basklh2.1new}(5) we have
$\mathcal T_0(f) Z = \mathcal T_0(f)(Z+Z^2) \neq 0$ 
since $\Lambda(Z^2, \mathcal T_0) \subseteq [2b,\infty)$ by Lemma \ref{basknewlh2.2}.
Consequently, $Z+Z^2\neq 0$ and $(I+Z)^2 \neq I+Z$.
\end{proof}

The operator $V$ from Example \ref{exa4} is memoryless if the function $v$ is constant. We also have  $V \in \mathscr C$ given ${\rm{supp}}\,\widehat v\in [0,\infty)$, $V \in \mathscr {A\!\!C}$ given ${\rm{ supp}}\,\widehat v\in (-\infty, 0]$, $V \in \mathscr {H\!\!C}_a$, given ${\rm{ supp}}\,\widehat v\in [a,\infty)$, and $V \in \mathscr {H\!\!A\!\!C}_a$ given ${\rm{ supp}}\,\widehat v\in (-\infty, -a]$; here $\widehat v$  is the tempered distribution that is the  Fourier transform of the function $v$ which is assumed to be bounded.

An operator from Example \ref{exa5} is memoryless, causal, anticausal, hypercausal, or hyperanticausal if its matrix is diagonal,  lower triangular,  upper triangular, strictly lower triangular, or strictly upper triangular, respectively.

More information on the abstract theory of causality can be found in \cite{BK05} and references therein.

\subsection{Periodic and almost periodic operators}\label{subsec3.2}

Let $\mathcal{M}$ be an $A$-harmonious space that is continuously and injectively embedded into  $\mathcal{M}_0$. The submodules of periodic and almost periodic operators   $\mathcal {AP\, M}$ and $\mathcal{P_\omega(M)}$, $\omega > 0$, from Definitions \ref{ap} and \ref{app}  are also $A$-harmonious spaces. An operator $V$ from Example \ref{exa4} is almost periodic ($\omega$-periodic) if the function $v$ is almost periodic ($\omega$-periodic). All operators in Example \ref{exa5} belong to $\mathcal P_{2\pi}(\mathcal M)$.

The spaces $\mathcal {AP\, M}$ and $\mathcal{P_\omega(M)}$, $\omega > 0$, are Banach algebras   (see, e.~g.,
\cite{BK10}). It is not hard to see that if
$\mathcal F$ is a closed submodule of $\mathcal{M}$, then the spaces $\mathcal {AP^F M}$ and $\mathcal{P^F_\omega(M)}$, $\omega > 0$, are also $A$-harmonious spaces and Banach algebras (see Definition \ref{inap}). An operator $V\in \mathcal {M = M}_0= \mathrm{End}\, L_2$ from Example \ref{exa4} is $\mathcal M_c$-almost-periodic if the function $v$ is almost periodic at infinity \cite{BSS19, BST18, B20}.

More information on almost periodic operators can be found in \cite{BK10, B04} and references therein. The notion of almost periodicity with respect to a submodule is new.

\subsection{Wiener and Beurling classes}

The classes presented in this section were introduced, for example, in \cite{BK14, BK18c}. To define them, we shall make use of the family of functions $(\phi_t)_{t\in\mathbb R}$ from $L_1(\mathbb R)$ given by
\begin{equation}\label{tri}
\widehat\phi(\lambda) = \widehat\phi_0(\lambda) = (1-|\lambda|)\chi_{[-1,1]}(\lambda)\quad\mbox{and}
\quad\widehat\phi_t(\lambda) = \widehat\phi(\lambda-t), \ \lambda, t\in\mathbb R.
\end{equation}
We note that instead of the functions $(\phi_t)_{t> 0}$ one can use other partitions of unity  \cite{BK14}, for example, those built from functions in \eqref{otrap}. The choice of a partition of unity does not affect the resulting class of operators.

Let $\mathcal{M}$ be an $A$-harmonious space that is continuously and injectively embedded into  $\mathcal{M}_0$ and  $\nu$ be a weight.
The Wiener class $\mathcal W_\nu(\mathcal M)$ consists of all operators $X\in\mathcal{M}$ satisfying
\[
\|X\|_{\mathcal W_\nu} = \int_{\mathbb R} \|\mathcal T_0(\phi_t)X\|_{\mathcal M}\nu(t)dt < \infty.
\]
From \cite{BK14, BK18c}, it follows that any Wiener class is an $A$-harmonious space. It is a Banach algebra if the weight $\nu$ is submultiplicative.

An operator $V\in \mathcal {M = M}_0= \mathrm{End}\, L_2$ from Example \ref{exa4} belongs to $\mathcal W_\nu(\mathcal M)$ if the function $v$ belongs to the Beurling algebra   $L_\nu(\mathbb R)$, i.~e.~$\int_{\mathbb R} |v(t)|\nu(t)dt < \infty$. An operator $X\in \mathcal M$ from Example \ref{exa5} belongs to $\mathcal W_\nu(\mathcal M)$ if its diagonals   $X_n$, $n\in\mathbb Z$, are $\nu$-summable, i.~e., $\sum_{n\in\mathbb Z} \nu(n)\|X_n\|_\mathcal M<\infty$.

The Beurling class $\mathcal B_\nu(\mathcal M)$ consists of all operator $X\in\mathcal{M}$ satisfying
\[
\|X\|_{\mathcal B} = \sum_{n\in\mathbb Z_+} \nu(n)\max_{|k|\ge n}\|\mathcal T_0(\phi_k)X\|_{\mathcal M} < \infty.
\]
From \cite{BK14}, it follows that Beurling classes are   $A$-harmonious spaces. Moreover,  $\mathcal B_\nu(\mathcal M)$ is a Banach algebra if the weight $\nu$ is submultiplicative. An operator $X\in \mathcal M$ from Example \ref{exa5} belongs to $\mathcal B_\nu(\mathcal M)$ with $\nu\equiv 1$ if its diagonals $X_n$, $n\in\mathbb Z$, satisfy
\[
 \sum_{n\in\mathbb Z_+} \max_{|k|\ge n}\|X_k\|_{\mathcal M} < \infty.
\]

\subsection{Jaffard classes}\label{jafclass}

The classes in this section were defined in \cite{BK18c}. As in the other examples of harmonious spaces, we assume that $\mathcal{M}$ is an  $A$-harmonious space that is continuously and injectively embedded into  $\mathcal{M}_0$. Additionally, we let $\nu$ be a weight and  $\omega$ be a bi-weight.

The Jaffard-$1$ class  $\mathfrak F^1_{\omega} = \mathfrak F^1_{\omega}(\mathcal M)$ is defined by
\[
\mathfrak F^1_{\omega} = \{X\in \mathcal{M}: \|X\|_{\mathfrak F^1_{\omega}} = \|X\|_{\mathcal{M}}+|X|_{\mathfrak F^1_{\omega}}  < \infty\},
\]
where
\[
|X|_{\mathfrak F^1_{\omega}} = \sup_{a,b\in\mathbb R} \omega(a,b)\|\mathcal T(\phi_a)X\mathcal T(\phi_b)\|_{\mathcal{M}} ,\quad X\in\mathcal M,
\]
and $\phi_a$, $a\in\mathbb R$, are given by \eqref{tri}.

The Jaffard-$2$ class  $\mathfrak F^2_{\nu}= \mathfrak F^2_{\nu}(\mathcal M)$ is defined by
\[
\mathfrak F^2_{\nu} = \{X\in \mathcal M: \|X\|_{\mathfrak F^2_{\nu}} = \|X\|_{\mathcal M}+|X|_{\mathfrak F^2_{\nu}}  < \infty\},
\]
where
\[
|X|_{\mathfrak F^2_{\nu}} = \sup_{a\in\mathbb R} \nu(a)\|\mathcal T_0(\phi_a)X\|_{\mathcal M},
\]
and $\phi_a$, $a\in\mathbb R$, are given by \eqref{tri}.

The Jaffard-$1$ and Jaffard-$2$ classes are $A$-harmonious spaces. In \cite{BK18c}, it was shown that these classes are Banach algebras if their weights are balanced and subconvolutive. It was also shown in \cite{BK18c} that if a weight is balanced and quasi-subconvolutive then the corresponding Jaffard classes are  $A$-admissible Banach modules in the sense of \cite[Definition 3.2]{BK18c} (see also Definition \ref{adop} below); in this case, however, they are not necessarily Banach algebras.

An operator $X\in \mathcal M$ from Example \ref{exa5} belongs to $\mathfrak F^1_{\omega}(\mathcal M)$ if its matrix elements $X_{mn}$, $m,n\in\mathbb Z$, satisfy
\[
\sup_{m,n\in\mathbb Z} \omega(m,n)\|X_{mn}\|_{\mathcal{M}} < \infty.
\]
An operator $X\in \mathcal M$ from Example \ref{exa5} belongs to  $\mathfrak F^2_\nu(\mathcal M)$   if its matrix diagonals $X_n$, $n\in\mathbb Z$, satisfy
\[
\sup_{n\in\mathbb Z} \nu(n)\|X_{n}\|_{\mathcal{M}} < \infty.
\]

The following class of $A$-harmonious spaces is also important (see \cite[Definition 2.15]{BK18c}).

\begin{definition}\label{defsubst}
An $A$-harmonious space $\mathcal M$ is substantial if it contains all Jaffard classes $\mathfrak F^2_{\nu}(\mathcal M_0)$, where $\nu(t) = M\gamma^{|t|}$ for some $M \ge 1$ and  $\gamma > 1$.
\end{definition}

\subsection{$\ABK$-classes}

The operator classes defined here generalize the ones from \cite{ABK08}, see also \cite{SS09}. As above, we assume that
$\mathcal{M}$ is an  $A$-harmonious space that is continuously and injectively embedded into  $\mathcal{M}_0$  and
$\nu$ is a weight. Given $a,b\in\mathbb R$, we also assume that the representation $\mathcal T_{ab}: \mathbb R \to \mathrm{End}\,\mathcal{M}$
 is well defined by $\mathcal T_{ab}(t)X = \mathcal T(at)X\mathcal T(-bt)$, $X\in \mathcal{M}$, and
$(\mathcal{M}, \mathcal T_{ab})$ is a non-degenerate Banach $L_1(\mathbb R)$-module.

The class $\ABK_{\!\!\nu}(a,b)$ consists of all operators $X\in\mathcal{M}$ which satisfy
\[
\|X\|_{\ABK_{\!\!\nu}(a,b)}=\int_{\mathbb R} \nu(t)\|\mathcal T_{ab}(\phi_t)X\|_{\mathcal M}dt < \infty,
\]
where $\phi_t$, $t\in\mathbb R$, are given by \eqref{tri}.

The class $\ABK^{\!\nu}(a,b)$ consists of all operators $X\in\mathcal{M}$ which satisfy
\[
\|X\|_{\ABK^{\!\nu}(a,b)}=\sup_{t\in\mathbb R} \nu(t)\|\mathcal T_{ab}(\phi_t)X\|_{\mathcal M} < \infty,
\]
where $\phi_t$, $t\in\mathbb R$, are given by \eqref{tri}.

An operator $VD_{-1}$ from Example \ref{exa4} belongs to $\ABK_{\!\!\nu}(1,-1)$ if the operator $V$  belongs to the Wiener class $\mathcal W_\nu(\mathcal M)$. An operator $VD_{-1}$  belongs to $\ABK^{\!\nu}(1,-1)$ if the operator $V$  belongs to the Jaffard-$2$ class $\mathfrak F^2_\nu(\mathcal M)$.

For operators in Example \ref{exa5}, classes $\ABK$ describe the decay of matrix elements in the direction of the vector  $(a,b)$.

\section{Abstract scheme of the method of similar operators}\label{basksec3}

The abstract method of similar operators possesses a rich historical lineage of development.  Its origins can be traced back to a diverse array of similarity and perturbation techniques, including
 the classical perturbation methods of celestial mechanics, Ljapunov’s kinematic similarity method \cite{GKK96, Lj56, N15}, Friedrichs’ method of similar operators developed primarily for application in quantum mechanics \cite{F65}, and Turner’s method of similar operators \cite{T65, U04}.
The method has undergone extensive refinement and been applied for various classes of unbounded linear operators. A notable body of work has contributed to its evolution and utilization, exemplified by references such as \cite{B83, B86, B94, B15, BDS11, BKR17, BKU18, BKU19, BP19, P21}.
 
As we mentioned in the introduction, similarity
transformations are widely used in various areas of algebra and analysis, starting with the matrix diagonalization. We cite \cite{SS19} as a comprehensive survey on the use of similarity transformations in perturbation theory of linear operators. We also recall an alternative name for the method of similar operators -- transmutation operator method  \cite{DL57}. One of the primary features that makes similarity transformations useful is the fact that similar operators have the same spectral properties. Thus, if it is possible to find a similarity transformation of an operator of interest into one, for which the spectral properties are known, the knowledge would extend to the original operator.

We begin our exposition of the abstract method with the definition of similar operators (in the unbounded case).

\begin{definition}[\cite{B87}]
Two linear operators $A_i: D(A_i)\subset\mathcal{X}\to\mathcal{X}$, $i=1, 2$, are called similar if there exists a continuously invertible operator $U\in\mathrm{End}\,\mathcal{X}$ such that
$UD(A_2)=D(A_1)$, $A_1Ux=UA_2x$, $x\in D(A_2)$. The operator $U$ is called the transformation operator of $A_1$ into $A_2$ or an intertwining operator.
\end{definition}

The key notion of the method of similar operators is that of an admissible (for the operator  $A:
D(A)\subset\mathcal{X}\to\mathcal{X}$) triple.

\begin{definition}[\cite{B87}]\label{baskdef3.3}
Let $\mathcal{M}$ be a linear space of operators from $\mathfrak{L}_A(\mathcal{X})$, $J: \mathcal{M}\to\mathcal{M}$, and
$\Gamma: \mathcal{M}\to\mathrm{End}\,\mathcal{X}$ be linear operators (transforms). The space $\mathcal{M}$ is called the space of admissible perturbations and the triple $(\mathcal{M}, J, \Gamma)$ is called admissible for the (unperturbed) operator $A$ if the following properties hold.
\begin{enumerate}
    \item $\mathcal{M}$ is a Banach space that is continuously embedded in $\mathfrak{L}_A(\mathcal{X})$; in particular,
    $\|X\|_{\mathcal M}\leqslant \mathrm{const}\,\|X\|_A$, $X\in\mathcal{M}$, where $\|\cdot\|_{\mathcal M}$ is the
    norm in $\mathcal{M}$.
    \item The transforms $J$ and $\Gamma$ are continuous linear operators.
    \item $(\Gamma X)D(A)\subset D(A)$ and $A(\Gamma X)-(\Gamma X)A=X-JX$, $X\in\mathcal{M}$.
    \item $X\Gamma Y$ and $(\Gamma X)Y\in\mathcal{M}$ for all $X$, $Y\in\mathcal{M}$, and there exists a constant $\gamma>0$ such that
$$
\|\Gamma\|\leqslant \gamma, \quad \max\{\|X\Gamma Y\|_{\mathcal{M}}, \|(\Gamma X)Y\|_{\mathcal{M}}\}\leqslant \gamma\|X\|_{\mathcal{M}}\|Y\|_{\mathcal{M}}, \quad X, Y\in\mathcal{M}.
$$
    \item For any $X\in\mathcal{M}$ and $\varepsilon>0$ there exists $\lambda_\varepsilon\in\rho(A)$ such that
$\|X(A-\lambda_\varepsilon I)^{-1}\|\leqslant\varepsilon$.
\end{enumerate}
\end{definition}

\begin{remark}
Property 5 in Definition \ref{baskdef3.3} may be replaced with any condition that would ensure $(\Gamma X)D(A)= D(A)$. For example, one may require $\mathrm{Ran}\,\Gamma X\subset D(A)$ and $A\Gamma X\in\mathrm{End}\,\mathcal{X}$
for any $X\in\mathcal{M}$.
\end{remark}

Typically, it is assumed in the method of similar operators that the transform $J$ is an idempotent. In this case, it is customary to also require that
\begin{equation}\label{dopusl}
J((\Gamma X)JY) = 0 \mbox{ for all } X,Y\in\mathcal M.
\end{equation}
In this paper, Definition \ref{baskdef3.3}
is more general than, for example, \cite[Definition~3.1]{BKU19}. Therefore, standard theorems of the method, such as e.g.~ \cite[Theorem~3.1]{BKU19}, are not immediately applicable. In particular, our next theorem requires a proof. To formulate the result, let us introduce the nonlinear transform
 $\Phi: \mathcal{M}\to\mathcal{M}$ given by
\begin{equation}\label{bask3.1}
\Phi(X)=B\Gamma X-(\Gamma X)(JX)+B, \ X\in\mathcal M,
\end{equation}
and the constant
\begin{equation}\label{bask3.1'}
j =\max\{1, \|J\|\}.
\end{equation}

\begin{theorem}\label{baskth3.1}
Let $(\mathcal{M}, J, \Gamma)$ be an admissible triple for an operator $A: D(A)\subset\mathcal{X}\to\mathcal{X}$ and
\begin{equation}\label{bask3.2}
(3+2\sqrt{2})j\gamma \|B\|_{\mathcal M}<1,
\end{equation}
where $\gamma$ comes from Definition \ref{baskdef3.3}(4), and the constant $j$ is given by \eqref{bask3.1'}. Then the transform
$\Phi: \mathcal{M}\to\mathcal{M}$ given by \eqref{bask3.1} is a contraction and has a unique fixed point $X_*$, i.~e.~$X_*=\Phi(X_*)$, in the ball
$$
\mathcal{B}=\{X\in\mathcal{M}: \|X-B\|_{\mathcal M}\leqslant \sqrt{2}\|B\|_{\mathcal M}\}.
$$
The operator $X_*$ is the limit of simple iterations $X_n=\Phi(X_{n-1})$, $n\geqslant 2$, $X_1=B$. Moreover, under the above conditions, the operator  $A-B$ is similar to the operator  $A-JX_*$ and the intertwining operator is given by
$U=I+\Gamma X_*\in\mathrm{End}\,\mathcal{X}$, i.~e.
\begin{equation}\label{bask3.3}
(A-B)(I+\Gamma X_*)=(I+\Gamma X_*)(A-JX_*).
\end{equation}
\end{theorem}

\begin{proof}
We shall use the Banach fixed point theorem \cite{B22, KA82}. 

We start by showing that $\Phi(\mathcal{B})\subseteq\mathcal{B}$ and
$\|\Phi(X)-\Phi(Y)\|_{\mathcal M}\leqslant q\|X-Y\|_{\mathcal M}$ for all $X$, $Y\in\mathcal{B}$ and some $q\in (0, 1)$.

Let 
$\mathcal{B}_\alpha=\{X\in\mathcal{M}: \|X-B\|_{\mathcal M}\leqslant \alpha\|B\|_{\mathcal M}\}$. For all $X \in\mathcal{B}_\alpha$,  
the inequalities
$$
\|X\|_{\mathcal M}\leqslant \|X-B\|_{\mathcal M}+\|B\|_{\mathcal M}\leqslant (\alpha+1)\|B\|_{\mathcal M},
$$
yield that invariance of the ball  $\mathcal{B}_\alpha$ for $\Phi$ would follow from the inequalities 
\begin{flalign*}
\|\Phi(X)-B\|_{\mathcal M} &\leqslant \|B\Gamma X-\Gamma X(JX)\|_{\mathcal M} \leqslant (\alpha+1)\gamma\|B\|_{\mathcal M}^2+(\alpha+1)^2\gamma\|B\|_{\mathcal M}^2\|J\| \\
&\leqslant j\gamma\|B\|_{\mathcal M}^2(\alpha+1+(\alpha+1)^2)\leqslant \alpha\|B\|_{\mathcal M}.
\end{flalign*}
The latter of the above inequalities holds if  $\alpha$ is such that
$$
\alpha^2j\gamma\|B\|_{\mathcal M}+\alpha(3j\gamma\|B\|_{\mathcal M} -1)+\alpha j\gamma\|B\|_{\mathcal M} <0.
$$
Clearly, such an $\alpha$ exists if $(3+2^{3/2})j\gamma\|B\|_{\mathcal M}<1$, which is precisely the condition in {\eqref{bask3.2}}. 

Additionally,  for all $X$, $Y\in\mathcal{B}_\alpha$, one has
\begin{flalign*}
\|\Phi(X)-\Phi(Y)\|_{\mathcal M} &\leqslant \|B\Gamma X-(\Gamma X)JX-B\Gamma Y+\Gamma Y(JY)\|_{\mathcal M}\\
&\leqslant \|B\Gamma(X-Y)+\Gamma Y J(Y-X)+\Gamma (Y-X)JX\|_{\mathcal M} \\
&\leqslant \gamma\|B\|_{\mathcal M}\|X-Y\|_{\mathcal M}+2(\alpha+1)\gamma \|J\|\|B\|_{\mathcal M}\|X-Y\|_{\mathcal M}\\
&\leqslant (2\alpha+3)j\gamma\|B\|_{\mathcal M}\|X-Y\|_{\mathcal M}.
\end{flalign*}
Consequently, the map $\Phi: \mathcal{M}\to\mathcal{M}$ is a contraction in  $\mathcal{B}_\alpha$ if $(2\alpha+3)j\gamma\|B\|_{\mathcal M}<1$.
From \eqref{bask3.2}, we derive $\alpha=\sqrt{2}$.
Thus, Banach fixed point theorem applies and we conclude that, in the ball $\mathcal{B}=\mathcal{B}_{\sqrt 2}$, the map $\Phi$ has a unique  fixed point 
$X_*\in\mathcal{B}\subset\mathcal{M}$, which can be found as a limit of simple iterations starting from $X_1=B$. This proves the first part of the theorem.

Next, we prove the equality \eqref{bask3.3}. We have
\begin{flalign*}
(A-B)(I+\Gamma X_*) &= A-B-A\Gamma X_*-B\Gamma X_*=A-B+(\Gamma X_*)A \\
&+ X-JX-B\Gamma X_* \\
&= (I+\Gamma X_*)A+X-JX-B\Gamma X_*-(X-B\Gamma X_*+\Gamma X_*(JX_*)) \\
&=(I+\Gamma X_*)(A-JX_*).
\end{flalign*}

We also need to prove that the intertwining operator $I+\Gamma X_*$ is continuously invertible. This follows from $\|\Gamma X_*\|\leqslant \gamma\|X_*\|_{\mathcal M}\leqslant
(\sqrt{2}+1)\gamma\|B\|_{\mathcal M}<1$, which yields that
$$
(I+\Gamma X_*)^{-1}=\sum_{n=0}^\infty(-1)^n(\Gamma X_*)^n.
$$

It remains to prove that $(I+\Gamma X_*)D(A)=D(A)$. From Definition \ref{baskdef3.3}(3), one immediately gets $(I+\Gamma X_*)D(A)\subseteq D(A)$. Let us show that $(I+\Gamma X_*)^{-1}D(A)\subseteq D(A)$. Using the same property of Definition \ref{baskdef3.3}, given $\lambda\in\rho(A)$, we get
\begin{flalign*}
\Gamma X_*(A-\lambda I)^{-1} &= (A-\lambda I)^{-1}(A-\lambda I)\Gamma X_*(A-\lambda I)^{-1} \\
&= (A-\lambda I)^{-1}(X_*-JX_*+\Gamma X_*A-\lambda\Gamma X_*)(A-\lambda I)^{-1} \\
&= (A-\lambda I)^{-1}((X_*-JX_*)(A-\lambda I)^{-1}+\Gamma X_*).
\end{flalign*}

Using Definition \ref{baskdef3.3}(5), we choose $\lambda\in\rho(A)$ such that
$$
\|(X_*-JX_*)(A-\lambda I)^{-1}+\Gamma X_*\|<1.
$$
Then
$$
(I+\Gamma X_*Ce)^{-1}(A-\lambda I)^{-1}=(A-\lambda I)^{-1}(I+(X_*-JX_*)(A-\lambda I)^{-1}+\Gamma X_*)^{-1},
$$
and the theorem is proved.
\end{proof}

\begin{remark}
We remark that initially \cite{B87} the method of similar operators with equation \eqref{bask3.1}
used the stronger assumption $6\|B\|j\gamma<1$ in place of \eqref{bask3.2}; moreover, the radius of the ball that was invariant for the map
$\Phi: \mathcal{M}\to\mathcal{M}$ was $2.5\|B\|$. We also note that the above method of proof is new as the proof in \cite{B87} did not use a fixed point theorem, it employed the method of majorant equations \cite{KA82}.
\end{remark}

Next, we present a modification of the method for the case when the transform
 $J:\mathcal{M}\to\mathcal{M}$ is an idempotent. As we mentioned above, for this case, we also add condition \eqref{dopusl}.
Then, applying the transform $J$ to both sides of the equation $X=\Phi(X)$, where $\Phi: \mathcal{M}\to\mathcal{M}$ is given by
\eqref{bask3.1}, we get $JX=J(B\Gamma X)+JB$, which yields the following replacement of \eqref{bask3.1}:
\begin{equation}\label{bask3.4}
X=B\Gamma X-(\Gamma X)JB-(\Gamma X)J(B\Gamma X)+B=:\Phi_1(X), \ X\in\mathcal M.
\end{equation}
This puts us in the setting of \cite{BKU19} and allows us to state the following result, which is essentially
\cite[Theorem~3.1]{BKU19}.
\begin{theorem}\label{baskth3.2}
Let $(\mathcal{M}, J, \Gamma)$ be an admissible triple for  $A: D(A)\subset\mathcal{X}\to\mathcal{X}$, where $J$ is an idempotent, $B\in\mathcal{M}$, and condition \eqref{dopusl} holds. Assume also that
\begin{equation}\label{bask3.5}
4j\gamma \|B\|_{\mathcal M}<1.
\end{equation}
Then the transform $\Phi_1:\mathcal{M}\to\mathcal{M}$ given by \eqref{bask3.4}  is a contraction and has a unique fixed point $X_*$ in the ball
$$
\mathcal{B}=\{X\in\mathcal{M}: \|X_*-B\|_{\mathcal M}\leqslant 3\|B\|_{\mathcal M}\}.
$$
The fixed point $X_*$ can be found via simple iterations $X_0=0$, $X_1=\Phi_1(X_0)=B$, \dots,
and the operator  $A-B$ is similar to the operator  $A-JX_*$ with the intertwining operator  given by
$I+\Gamma X_*\in\mathrm{End}\,\mathcal{X}$.
\end{theorem}

Condition \eqref{bask3.5} of Theorem \ref{baskth3.2} can be improved if the operator $B\in\mathcal{M}$ 
satisfies $JB=0$. We derive the following result from  \cite[Remark 3.1]{BKU19}.

\begin{theorem}\label{baskth3.2c}
Assume that $(\mathcal{M}, J, \Gamma)$ is an admissible triple for $A:$ $D(A)\subset\mathcal{X}\to\mathcal{X}$,
$B\in\mathcal{M}$, condition \eqref{dopusl} holds,  $JB=0$, and
$$
3j\gamma \|B\|_{\mathcal M}<1.
$$
Then the operator $A-B$ is similar to the operator  $A-JX_*$ with the intertwining operator  given by
$I+\Gamma X_*\in\mathrm{End}\,\mathcal{X}$, where
 $X_*\in\mathcal{M}$ is the solution of the nonlinear equation
$$
X=B\Gamma X-(\Gamma X)J(B\Gamma X)+B=\Phi_2(X).
$$
\end{theorem}

We conclude this section by outlining the setting in which the method of similar operators reduces to its predecessor - Friedrichs' method \cite{B87, DS88III, F65p,S83}. In this case, the transform $J$ satisfies $J = 0$. Let us state and prove the corresponding result.

\begin{theorem}\label{baskth3.4}
Let $(\mathcal{M}, J, \Gamma)$ be an admissible triple with $J=0$. Assume that either    
\begin{equation}\label{bask3.6}
2\gamma\|B\|_{\mathcal M}<1,
\end{equation}
or $\mathcal{M}$ is a space of quasi-nilpotent operators that is invariant for $\Gamma$ and
\begin{equation}\label{bask3.6up}
\gamma\|B\|_{\mathcal M}<1.
\end{equation}
Then the operator $A-B$, $B\in\mathcal{M}$, is similar to the operator 
$A$ and 
$$
(A-B)(I+\Gamma X_*)=(I+\Gamma X_*)A,
$$
where $X_*\in\mathcal{M}$ is the unique solution of the operator equation
\begin{equation}\label{bask3.7}
X=B\Gamma X+B=\Phi_3(X),
\end{equation}
which can be found by simple iterations starting with $X_{1}=B$.
\end{theorem}
\begin{proof}
The proof follows the same blueprint as that of
Theorem \ref{baskth3.1}. We will only point out a few key points: constructing a ball where the fixed point theorem holds and ensuring that the intertwining operator $I+\Gamma X_*$ is invertible. 

Thus, we consider the ball
 $\mathcal{B}_\alpha=\{X\in\mathcal{M}: \|X-B\|_{\mathcal M}\leqslant \alpha\|B\|_{\mathcal M}\}$, where $\alpha\in\mathbb{R}_+$,
and find a condition on $\alpha$ which ensures that the ball is invariant for $\Phi_3:\mathcal{M}\to\mathcal{M}$; the map $\Phi_3$ will automatically be a contraction in the ball due to \eqref{bask3.6up}. In other words, we want to find for which $\alpha$ one has
$\|\Phi_3(X)-B\|_{\mathcal M}\leqslant \alpha\|B\|_{\mathcal M}$ whenever $\|X-B\|_{\mathcal M}\leqslant \alpha\|B\|_{\mathcal M}$. We use a weaker condition: $\|X\|_{\mathcal M}\leqslant (\alpha+1)\|B\|_{\mathcal M}$.
In that case $\|\Phi_3(X)-B\|_{\mathcal M}=\|B\Gamma X\|_{\mathcal M}\leqslant \gamma(\alpha+1)\|B\|_{\mathcal M}^2$, and we see that
$\gamma\|B\|_{\mathcal M} \leqslant \frac{\alpha}{\alpha+1}$ is sufficient for the invariance. As the right-hand side tends to $1$ when $\alpha\to\infty$,
we get that \eqref{bask3.6up} yields existence of $\alpha$ for which $\mathcal{B}_\alpha$ is invariant.

Next, we show that the operator $I+\Gamma X_*$ is invertible. In case of \eqref{bask3.6}, the above argument showed that $X_* \in \mathcal{B}_\alpha$ with $\alpha = 1$, yielding $\|X_*\|_{\mathcal M}\leqslant 2\|B\|_{\mathcal M}$ and, therefore, $\|\Gamma X_*\| \le \gamma \|X_*\|_{\mathcal M} \le 2\gamma\|B\|_{\mathcal M}< 1$. Under the alternative condition, the operator $\Gamma X_*$ is quasi-nilpotent and, therefore, $I+\Gamma X_*$ is, indeed, invertible.

The remainder of the proof is analogous to that of Theorem \ref{baskth3.1}. 
\end{proof}

\section{Admissible triples in harmonious spaces}\label{basksec4}

In this section, we construct a specific family of admissible triples and restate the general theorems of the previous section for this family. Here, we pick a space of admissible perturbations from a class of certain $A$-harmonious spaces that is continuously embedded into the module $\mathcal{M}_0$ from Section \ref{basksec2}. Let us reiterate that all $A$-harmonious spaces in this section are assumed to be continuously and injectively embedded
into $\mathcal{M}_0$, even if we do not explicitly state this condition below.

\begin{definition}\label{adop}
Let $(\mathcal{M}, J, \Gamma)$ be an admissible triple for operator $A$ and assume that $\mathcal{M}$ is
an $A$-harmonious space  that is continuously and injectively embedded into the module $\mathcal{M}_0$. The space $\mathcal{M}$  is called $A$-admissible for the transforms  $J: \mathcal{M}\to\mathcal{M}$ and
$\Gamma: \mathcal{M}\to\mathrm{End}\,\mathcal{X}$ if
Condition 4 of Definition \ref{baskdef3.3} holds.
\end{definition}

We note that an $A$-harmonious space $\mathcal{M}$  is automatically $A$-admissible if $\Gamma: \mathcal{M\to M}$ is a bounded
operator and $\mathcal{M}$ is a Banach algebra. In particular, the module $\mathcal M_0$ itself is  $A$-admissible.

Let $\mathcal{M}$ be an $A$-harmonious space. We define the transforms $J =J_a: \mathcal{M}\to\mathcal{M}$ and
$\Gamma = \Gamma_a: \mathcal{M}\to\mathcal{M}$, $a>0$, via
\begin{equation}\label{bask4.1}
\Gamma X=\mathcal T_0(\psi_a)X \quad \mbox{and}\quad JX=\mathcal T_0(\varphi_a)X, \quad X\in\mathcal{M},
\end{equation}
where the functions $\varphi_a$ and $\psi_a$, $a>0$, were defined by \eqref{otrap} and \eqref{psia}, respectively. 
Then, if $\mathcal{M}$ is $A$-admissible for $J_a$ and $\Gamma_a$ for each $a>0$,
then it is an $A$-admissible module in the sense of \cite[Definition 3.2]{BK18c}. In particular, Condition (2) from
\cite[Definition 3.2]{BK18c} holds due to boundedness of the representation $\mathcal T_0: \mathbb{R}\to\mathrm{End}\,\mathcal{M}$,
and Condition (3) -- due to the fact that the constants $\gamma=\gamma_a$ in Condition 4 of Definition \ref{baskdef3.3} satisfy
\begin{equation}\label{galim}
\lim_{a\to \infty} \gamma_a \leqslant \lim_{a\to \infty}  \|\psi_a\|_1\leqslant \lim_{a\to \infty}  1.35/a = 0
\end{equation}
in view of the estimates \eqref{psiest}. 

\begin{definition}[\cite{BK18c}]\label{asdop}
An $A$-admissible space $\mathcal{M}$ is called strictly $A$-admissible if
$\mathcal T(t)X \in\mathcal M$ (and, consequently, $X\mathcal T(t) \in\mathcal M$) for all $t\in\mathbb R$ and $X\in\mathcal M$.
\end{definition}

Not all $A$-harmonious spaces are $A$-admissible and not all $A$-admissible spaces are strictly $A$-admissible (for the transforms $J$ and $\Gamma$ in \eqref{bask4.1}). 

\begin{theorem}\label{baskth4.1}
Let $\mathcal{M}$ be an $A$-admissible space for the transforms $J$ and $\Gamma$ given by \eqref{bask4.1}.
Then $(\mathcal{M}, J, \Gamma)$
is an admissible triple for the operator $A$.
\end{theorem}

\begin{proof}
We need to check Conditions 1--5 of Definition \ref{baskdef3.3}.

Since
$\mathcal{M}$ is continuously embedded into $\mathcal{M}_0$,
Condition 1 follows.

Continuity of the transforms $J$ and $\Gamma$, which is required by Condition 2, is immediate from their definition: indeed, we have $\|J\|\leqslant \|\varphi_a\|_1$ and $\|\Gamma\|\leqslant \|\psi_a\|_1$. 

Condition 3 of Definition \ref{baskdef3.3} follows from Lemma \ref{basklh2.1}.

Condition 4 is built into the definition of $A$-admissibility of the space $\mathcal{M}$.

Finally, Condition 5 is implied by Lemma \ref{resgen}, and the theorem is proved.
\end{proof}

The above result allows us to formulate the following consequences of the general theorems in the previous section.

\begin{theorem}\label{baskth4.2}
There exists $a\in\mathbb{R}_+$ such that an operator $A-B$, $B\in\mathcal{M}$, is similar to the operator $A-\mathcal T_0(\varphi_a)X_*$, where the operator $X_*\in\mathcal{M}$
is the unique solution of the equation \eqref{bask3.1} that can be found via simple iterations.
\end{theorem}

Theorem \ref{baskth4.2} is an immediate corollary of Theorems \ref{baskth4.1} and
\ref{baskth3.1} in view of condition \eqref{galim}.
We note that the proof uses the inequalities $\|\Gamma\|\leqslant \|\psi_a\|_1\leqslant
1.35a^{-1}$ and $\|J\|\leqslant \|\varphi_a\|\leqslant \sqrt{3}$. Thus Condition \eqref{bask3.2} is satisfied for sufficiently large
 $a$.
\begin{remark}
Theorem \ref{baskth4.2} can also be deduced from Theorem 23.1 in \cite{B87}.
\end{remark}

The main advantage of studying the operator $A-\mathcal T_0(\varphi_a)X_*=A-JX_*$ instead of $A-B$ is the fact that the operator
$\mathcal T_0(\varphi_a)X_*$ has a compact Beurling spectrum $\Lambda(\mathcal T_0(\varphi_a)X_*, \mathcal T_0)$ (see
Lemma \ref{basklh2.1new} and Corollary \ref{cnormest1}).

\begin{corollary}\label{maincor}
An operator $A-B$, with $B\in\mathcal{M}$, is similar to an operator
$A-C$, where $C\in\mathcal{M}$ and $\Lambda(C,\mathcal T_0)$ is compact. Consequently, the function $\tau_{C}$
(see Definition \ref{baskdef2.3}) admits a holomorphic extension to an entire function of exponential type.
\end{corollary}

The above corollary can be used to prove important results on spectral invariance of operators $A-B$, such as \cite[Theorem 3.4]{BK18c}, a reformulation of which is presented next. 
By $\sigma_{\mathcal M}(X)$, $X\in \mathfrak{L}_A(\mathcal{X})$, we denote the set
$
\sigma_{\mathcal M}(X) = \{\lambda \in\mathbb C: (X-\lambda I)^{-1}\notin\mathcal M\}.
$
The formulation below uses the notion of $A$-substantial spaces that can be found in \cite[Definition 2.15]{BK18c}. A module $\mathcal M$ is $A$-substantial
if it contains all operators with exponential memory decay (see \cite{BK14}), or, equivalently, all Jaffard-2 classes $\mathfrak F_\nu^2(\mathcal{M}_0)$ with an exponential weight $\nu$ (see Subsection \ref{jafclass}).

\begin{theorem}\label{baskth4.3}
Let $\mathcal M$ be an $A$-substantial $A$-admissible space  for transforms $J$ and $\Gamma$ in \eqref{bask4.1}. Then $
\sigma(A-B) = \sigma_{\mathcal M}(A-B),$
$B\in\mathcal{M}$.
\end{theorem}

\begin{proof}
Despite certain discrepancies in the definitions between \cite{BK18c} and the current paper, the proof of \cite[Theorem 3.4]{BK18c} applies without any substantial changes. We present the outline of the proof for completeness. 

The similarity between operators in Corollary \ref{maincor} implies that for $\lambda \in\rho(A-B)$ we have
\[
(A-B-\lambda I)^{-1} = (I+\Gamma X_*)^{-1}(A-C-\lambda I)^{-1}(I+\Gamma X_*).
\]
We then have $I+\Gamma X_*\in\mathcal M$ because $X_*\in \mathcal M$ and $(I+\Gamma X_*)^{-1}\in\mathcal M$ because $\|\Gamma X_*\|_{\mathcal M} < 1$.
Inclusion $(A-C-\lambda I)^{-1}\in\mathcal M$ is implied by compactness of $\Lambda(C,\mathcal T_0)$ 
and the results of \cite{BK14}, which are applicable since the module $\mathcal M$ is $A$-substantial. Finally, we use $A$-admissibility of the module $\mathcal M$ and the representation of $(I+\Gamma X_*)^{-1}$as a Neumann series, to conclude that $(A-B-\lambda I)^{-1}\in\mathcal M$.
\end{proof}

Spectral invariance results such as the one above have been used extensively in various areas of analysis and applications. For example, in differential equations, such results lead to existence of solutions in a certain class \cite[and references therein]{B03, BK14, BSS19}, whereas in frame theory, spectral invariance is the key tool for studying localization of dual frames \cite[and references therein]{ABK08, BCHL06I, BCKOR14, BK10,  G04, SS09}. 

As in the previous section, we proceed to present stronger similarity results that can be obtained under additional conditions.

\begin{theorem}\label{baskth5.1}
Assume that
\begin{equation}\label{specond}
\Lambda(\mathcal {M, T}_0)\cap\left((-2a,-a)\cup(a, 2a)\right)=\varnothing,
\end{equation}
and the number $a$ and the operator $B$ yield \eqref{bask3.5}, i.e.
\begin{equation}\label{prob5.4}
\frac{5.4\sqrt{3}}{a}\|B\|_{\mathcal M}<1.
\end{equation}
Then the operator $A-B$ is similar to the operator $A-JX_*=A-\mathcal T_0(\varphi_a)X_*$, $X_*\in\mathcal{M}$, where the operator $JX_*$ satisfies  $\Lambda(JX_*,\mathcal T_0)\subseteq [-a,a]\cap \Lambda(\mathcal {M, T}_0)$. In particular, if $\Lambda(\mathcal {M, T}_0)\cap[-a,a] = \{0\}$, then $JX_*$ is memoryless. As before $X_*$ is the unique solution of the equation \eqref{bask3.1} and the intertwining operator of $A-B$ into
$A-JX_*$ is  the invertible operator $A-\mathcal T_0(\psi_a)X_*$.
\end{theorem}
\begin{proof}
Condition \eqref{specond} via Lemmas \ref{basklh2.1'new} and \ref{basknewlh2.2} guarantees that the transform $J = \mathcal T_0(\varphi_a)$ is an idempotent in $\mathrm{End}\,\mathcal{M}$ and that condition \eqref{dopusl} holds. Therefore, it remains to apply Theorems \ref{baskth4.1} and \ref{baskth3.2} to complete the proof.
\end{proof}

\begin{corollary}\label{baskth5.c1}
Assume that $\mathcal {M =  P}_\omega$ is the space of periodic operators from Definition \ref{app}, where $\omega\le\pi/ a$,
and condition \eqref{prob5.4} holds. Then the operator $A-B$, $B\in\mathcal{P}_\omega$, is similar to the operator $A-JX_*$, where
$X_*\in\mathcal{P}_\omega$ is the unique solution of the nonlinear equation \eqref{bask3.1} and the operator $JX_*$ is memoryless.
\end{corollary}

Theorems \ref{baskth3.2c} and \ref{baskth4.1} immediately yield the following result.

\begin{corollary}\label{baskth5.c2}
In addition to the assumptions of Theorem \ref{baskth5.1}, suppose that $\mathcal T_0(\varphi_a)B = 0$ and
\begin{equation}\label{prob1}
\frac{4.05\sqrt{3}}{a}\|B\|_{\mathcal M}<1.
\end{equation}
Then the operator $A-B$, $B\in\mathcal{M}$, is similar to the operator $A-\mathcal{T}_0(\varphi_a)X_*=A-JX_*$, $X_*\in\mathcal{M}$,
where $JX_*$ satisfies $\Lambda(JX_*, \mathcal{T}_0)\subseteq [-a, a]\cap \Lambda(\mathcal{M}, \mathcal{T}_0)$.
In particular, if $\Lambda(\mathcal{M}, \mathcal{T}_0)\cap [-a, a]=\{0\}$, then $JX_*$ is memoryless.
\end{corollary}

Condition \eqref{specond} in Theorem \ref{baskth5.1} is difficult to check in practice. In view of the containment
$\Lambda(\mathcal {M, T}_0)\subseteq\Lambda(\mathcal {M}_0,\mathcal{ T}_0)$, which holds in any $A$-admissible space
$\mathcal {M}$, we can use Property 3 of Lemma \ref{basknewlh2.2} to obtain a condition that is stronger but is easy to verify.

\begin{propos}\label{baskth5.c}
Condition \eqref{specond}  in Theorem \ref{baskth5.1} is implied by
\begin{equation}\label{speconda}
 (\sigma(A)-\sigma(A)) \cap\left((-2a,-a)\cup(a, 2a)\right)=\varnothing.
 \end{equation}
\end{propos}

\begin{remark}
Observe that condition \eqref{speconda}  holds, for example, when   $\sigma(A)=\bigcup_{n\in\mathbb N} \sigma_n$ and the spectral components  $\sigma_n$ satisfy
 \[
 \mathrm{diam}\, (\sigma_n) \le a \quad\mbox{and}\quad \mathrm{dist}\, (\sigma_m, \sigma_n) \ge 2a,
 \quad m,n\in\mathbb N,\ m\neq n.
 \]
If the spectral components $\sigma_n$ are singletons, Proposition \ref{baskth5.c} yields similarity of the perturbed operator
 $A-B$ to a memoryless operator. We note that in this case memoryless operators have a block diagonal matrix (see \cite{BK05, BKU19}). 
\end{remark}

\begin{remark}
Looking back at Lemma \ref{basklh2.1}, observe that to apply the method of similar operators we need equality \eqref{bask2.1} to hold 
only on the Beurling spectrum $\Lambda(\mathcal{M}, \mathcal T_0)$; it need not hold for all $\lambda\in\mathbb{R}\setminus\{0\}$. Thus, if $\Lambda(\mathcal{M}, \mathcal T_0)\cap(-a,a)
\subseteq \{0\}$, the functions $\varphi_{a/2}$ and $\psi_{a/2}$ can be replaced by any functions $\widetilde{\varphi}_a$ and $\widetilde{\psi}_a$
that satisfy \eqref{bask2.1} outside of $(-a,a)$. In this case, one typically \cite{B83} uses the functions
$\widetilde{\tau}_a=\widehat{\widetilde{\varphi}}_a$ and
$\widetilde{\omega}_a=\widehat{\widetilde{\psi}}_a$, given by
$$
\widetilde{\tau}_a(\lambda)=
\begin{cases}
0, \quad &|\lambda|>a, \\
1-\frac{|\lambda|}{a}, \quad &|\lambda|\leqslant a,
\end{cases} \quad
\widetilde{\omega}_a(\lambda)=
\begin{cases}
\frac{1}{\lambda}, \quad &|\lambda|>a, \\
\frac{\lambda}{a^2}, \quad &|\lambda|\leqslant a.
\end{cases}
$$
These functions were used extensively in \cite{A65}.
\end{remark}

Let us now illustrate Theorem \ref{baskth3.4} using hypercausal operators.
It is easy to see that for any $a> 0$ the spaces $\mathscr {H\!\!C}_a$ and $\mathscr {H\!\!A\!\!C}_a$ are strictly $A$-admissible and $J = \mathcal T_0(\varphi_{a}) = 0$ for both $\mathscr {H\!\!C}_{2a}$ and $\mathscr {H\!\!A\!\!C}_{2a}$.

\begin{theorem}\label{baskth5.2}
Assume that $\mathcal M \in \{\mathscr {H\!\!C}_{2a},  \mathscr {H\!\!A\!\!C}_{2a}\}$ for some $a>0$ and the operator $B$ satisfies 
\begin{equation}\label{prob}
\|B\|_{\mathcal M}< a
\end{equation}
(which yields \eqref{bask3.6} due to Lemma \ref{hcest}).

Then the operator $A-B$ is similar to the operator $A$, and the intertwining operator of
 $A-B$ into $A$ is given by the invertible operator
$A-\mathcal T_0(\psi_a)X_*$, where $X_*\in\mathcal{M}$ is the unique solution of \eqref{bask3.7}.
\end{theorem}

We note that the method of similar operators for a certain class of causal operators appeared in \cite{S83} (see also \cite{B83}). However, in the current setting of causality, which was developed in \cite{BK05}, hypercausal operators need not be quasi-nilpotent (which is also the reason for why we could not use the weaker condition \eqref{bask3.6up} in the above theorem). Nevertheless, based on the techniques in  \cite{B83, BK05, S83}, we prove the following much stronger result that shows that condition \eqref{prob} in the above theorem is completely unnecessary.

\begin{theorem}\label{hypcaussim}
Assume $\mathcal M \in \{\mathscr {H\!\!C}_{2a},  \mathscr {H\!\!A\!\!C}_{2a}\}$ for some $a>0$. Then operator $A-B$, $B\in\mathcal M$, is similar to the operator $A$, and the similarity transform of $A-B$ into $A$ is given by an invertible  operator 
$I+\Gamma X_*$, where $\Gamma = \mathcal T_0(\psi_a)$ and $X_*\in\mathcal{M}$ is given by an absolutely convergent series 
\begin{equation}\label{causer}
   X_*= B + B\Gamma B + B\Gamma(B\Gamma B)+\ldots. 
\end{equation}
\end{theorem}

\begin{proof}
We need to establish two things: first, that the series in \eqref{causer} does, indeed, converge absolutely and second, that the operator $I+\Gamma X_*$ is invertible. The intertwining equality showing that $I+\Gamma X_*$ is indeed the similarity transform of $A-B$ into $A$ is established via a computation similar
to that in the proof of Theorem \ref{baskth3.1}.

We begin the proof of the first of the above assertions by letting $B_n$ be the $n$-th term of the series in \eqref{causer}. Then $B_n = B\Gamma B_{n-1} = B\mathcal{T}_0(\psi_a)B_{n-1}$ so that, by Lemma \ref{basknewlh2.2}(2), 
$\Lambda(B_n)\subseteq [2na,\infty)$ in the case of $\mathcal{M} = \mathscr {H\!\!C}_{2a}$ and of $\Lambda(B_n)\subseteq (-\infty, -2na]$ in the case of $\mathcal{M} =\mathscr {H\!\!A\!\!C}_{2a}$.
Repeatedly applying Lemma \ref{hcest}, we then compute that  
$\|B_n\|\le\frac{\|B\|^n}{(2a)^{n-1}(n-1)!}$. Hence, the series converges and we have
\begin{equation}\label{uminus}
    (I+\Gamma X_*)A = (A-B)(I+\Gamma X_*) \quad \mbox{on } D(A),
\end{equation}
via the same computation as in the proof of Theorem \ref{baskth3.1}.

To prove the second  assertion, let us construct the inverse for $I+\Gamma X_*$. We begin by  observing that  
 an  argument similar to the above establishes that an operator $X_\dagger\in \mathcal{M}$ is well defined by the absolutely convergent series
$$
X_\dagger = B - (\Gamma B)B +\Gamma((\Gamma B)B)B- \ldots
$$
and
\begin{equation}\label{uplus}
    A(I-\Gamma X_\dagger) = (I-\Gamma X_\dagger)(A-B) \quad \mbox{on } D(A).
\end{equation}
We claim that $(I-\Gamma X_\dagger)$ is the inverse of $(I+\Gamma X_*)$. Indeed, from \eqref{uminus} and \eqref{uplus}, we immediately get that
$$
A(I-\Gamma X_\dagger)(I+\Gamma X_*) = (I-\Gamma X_\dagger)(I+\Gamma X_*)A \quad \mbox{on } D(A),
$$
i.~e., the operator $Y = (I-\Gamma X_\dagger)(I+\Gamma X_*) = I - \Gamma X_\dagger+ \Gamma X_* - (\Gamma X_\dagger)(\Gamma X_*)$ commutes with the generator $A$ (on $D(A)$). It follows from Lemma \ref{comlemma} that $Y \in \mathscr{M}$ and, therefore,
$$
Y - I = - \Gamma X_\dagger+ \Gamma X_* - (\Gamma X_\dagger)(\Gamma X_*) \in \mathscr{M}\cap \mathcal{M} =\{0\}.
$$
Thus,  $(I-\Gamma X_\dagger)(I+\Gamma X_*) = I$. It follows that $(I-\Gamma X_*)(I+\Gamma X_\dagger)$ is an idempotent and Lemma \ref{causidem} yields $(I-\Gamma X_*)(I+\Gamma X_\dagger) = I$. Thus, the final required assertion is established and the proof is complete.
\end{proof}

\begin{remark}
    In the setting of this section, all of the examples for Theorem \ref{baskth3.4} that we know can be reduced to the case of hypercausal operators.
    In view of the above result, we conjecture that for $A$-admissible modules $\mathcal M$ such that $J = \mathcal{T}(\varphi_a) = 0$, we have that all operators $A-B$, $B\in\mathcal{M}$, are similar to $A$ without any additional assumptions on $B$. 
\end{remark}

\section{Application to first order abstract differential operators}\label{basksec5}

By far the most common application of the method of similar operators is in the study of spectral properties of differential operators. The similarity allows one to better understand  solutions of initial value problems (IVP) with perturbed differential operators. More precisely,
given an abstract differential equation $\mathfrak{L}x=f$ in a harmonious space $\mathcal{X}$ and a similarity transform $U$ of the operator $\mathfrak{L}$ into a simpler operator $\mathfrak{L}_0$, one reduces the investigation of the original equation to the equation $\mathfrak{L}_0u=h$, where $u=Ux$ and $h=Uf$. The latter equation is typically much easier to study due to the known structure of the operator $\mathfrak{L}_0$.

To illustrate the above idea, we consider operators $\mathfrak{L}$ of the form $\mathfrak{L}=A-V=-i\frac d{dt}-V: D(A)\subset\mathcal{X}\to\mathcal{X}$, where $V\in\mathcal M\subseteq \mathrm{End}\,\mathcal{X}$ is an operator of multiplication by an operator-valued function $v$ such that $V$ belongs to some $A$-admissible module $\mathcal M$. The method of similar operators allows us to replace an IVP with the operator $\mathfrak{L}$ by an IVP with a simpler potential function. In some cases (e.g., when we use Theorem \ref{baskth5.2}), this reduction turns a non-autonomous IVP into an autonomous one. In other cases (e.g., when we use Theorem \ref{hypcaussim}) the operator $\mathfrak{L}$ is replaced with the unperturbed operator $A=-i\frac d{dt}$. The remainder of this section explains this paragraph in detail.

We begin with a precise definition of the operators involved. The operator $A=-i\frac d{dt}$ is the generator of the
$L_1(\mathbb{R})$-module $(\mathcal{X}, S)$, where the representation $S$ is defined by \eqref{basknew1} (see Definition \ref{baskdef2.8}).
The resolvent of $A$ is given by \eqref{resf}, which allows us to write $D(A)=\mathrm{Im}\,R(\lambda, A)$, $\lambda \in\rho(A)$. Since the operator
$V$ is assumed to be bounded, we have $D(\mathfrak{L})=D(A)$. From Lemma \ref{equigen}, we deduce that $x\in\mathcal{X}$ belongs to the domain $D(A)$ of the operator $A$ if there exists
$y\in\mathcal{X}$ such that for real numbers $s\leqslant t$ one has
$$
x(t)=x(s)+i\int_s^ty(\tau)\,d\tau.
$$

The harmonious Banach space $\mathcal{X}$ in this section is a homogeneous Banach space from Definition \ref{baskdef26} below. It turns out to be a Banach 
$L_1(\mathbb{R})$-module whose structure is associated with the isometric representation $S$
given by \eqref{basknew1}. We follow \cite{BK16} as we define homogeneous spaces and provide a few examples.

Let $\mathscr{X}$ be a complex Banach space and $L_{1, s}=L_{1, s}(\mathbb{R}, \mathrm{End}\,\mathscr{X})$ be the space of all functions $F: \mathbb{R}\to\mathrm{End}\,\mathscr{X}$, that have the following properties:
\begin{enumerate}
    \item For each $x\in\mathscr{X}$ the function $s\mapsto F(s)x: \mathbb{R}\to\mathscr{X}$ is measurable.
    \item There exists a function $f\in L_1(\mathbb{R})$ such that
\begin{equation}\label{bask5.1}
\|F(s)\|\leqslant f(s).
\end{equation}
\end{enumerate}
For $F\in L_{1, s}$ we let
$
\|F\|
=\inf\|f\|,
$
where the infimum is taken over all functions $f\in L_1(\mathbb{R})$ satisfying \eqref{bask5.1}.
The space $L_{1, s}$ is then a Banach algebra with respect to the convolution 
$$
(F_1*F_2)(t)x=\int_{\mathbb{R}}F_1(s)F_2(t-s)x\,ds, \quad F_1, F_2\in L_{1, s},\ x\in\mathscr X,
$$
and $\|F_1*F_2\|\leqslant \|F_1\|\|F_2\|$.

To define homogeneous spaces of functions we need the following two vector spaces. By $L_{1, loc}= L_{1, loc}(\mathbb{R}, \mathscr{X})$ we denote the space of locally summable Bochner measurable equivalence classes of $\mathscr{X}$-valued functions. In particular, 
for a compact set $K\subset\mathbb{R}$ and $f\in L_{1, loc}$ we have
$$
\int_K\|f(t)\|_{\mathscr{X}}\,dt<\infty.
$$
By $S_p=S_p(\mathbb{R}, \mathscr{X})$, $p\in [1, \infty)$, we denote the Stepanov space 
\cite{LZ82} that consists of all functions $f\in L_{1, loc}$ such that
$$
\|f\|_{S_p}=\sup_{s\in\mathbb{R}}\Big(\int_0^1\|f(t+s)\|_{\mathscr{X}}^p\,dt\Big)^{1/p}, \quad p\in [1, +\infty).
$$
The norm $\|\cdot\|_{S_p}$ turns $S_p$ into a Banach space. 

\begin{definition}[\cite{BK16}]\label{baskdef26}
A Banach space $\mathcal{X}=\mathcal{X}(\mathbb{R}, \mathscr{X})$ of functions on $\mathbb{R}$ with values in a complex Banach space $\mathscr{X}$ is called homogeneous if the following conditions hold:
\begin{enumerate}
    \item The space $\mathcal{X}$ is continuously embedded into $S_1$.
    \item The representation $S: \mathbb{R}\to\mathrm{End}\,\mathscr{X}$, given by \eqref{basknew1}, is an isometric representation of the group $\mathbb{R}$ by operators in $\mathrm{End}\,\mathscr{X}$.
    \item For $x\in\mathcal{X}$ and $C\in\mathrm{End}\,\mathscr{X}$ the function $y(t)=C(x(t))$ belongs to $\mathcal{X}$ and $\|y\|\leqslant \|C\|\|x\|$.
    \item For $x\in\mathcal{X}$ and $F\in L_{1, s}$ the convolution
    $$
    (F*x)(t)=\int_{\mathbb{R}}F(s)x(t-s)\,ds
    $$
    belongs to $\mathcal{X}$ and $\|F*x\|\leqslant \|F\|\|x\|$.
    \item If some $x\in\mathcal{X}$ satisfies $f*x=0$ for each $f\in L_1$, then $x=0$.
\end{enumerate}
\end{definition}

The following Banach spaces are homogeneous or admit an equivalent norm that makes them homogeneous.

\begin{enumerate}
    \item Stepanov spaces $S_p=S_p(\mathbb{R}, \mathscr{X})$, $p\in [1, \infty)$.
    \item Bochner-Lebesgue spaces $L_p=L_p(\mathbb{R}, \mathscr{X})$, $p\in [1, \infty]$.
    \item The space $C_b=C_b(\mathbb{R}, \mathscr{X})$ of bounded continuous $\mathscr{X}$-valued functions with the norm
    $$
    \|x\|_\infty=\sup_{t\in\mathbb{R}}\|x(t)\|, \quad x\in C_b.
    $$
    \item The subspace $C_{ub}\subset C_b$ of uniformly continuous functions.
    \item The subspace $\mathcal{P}_\omega=\mathcal{P}_\omega(\mathbb{R}, \mathscr{X})$ of $\omega$-periodic $\mathscr{X}$-valued functions, $\omega\in\mathbb{R}$ 
    (see Subsection \ref{subsec3.2}).
    \item The space $\mathcal{AP}(\mathbb{R}, \mathscr{X})$ of Bohr almost periodic functions (see Subsection \ref{subsec3.2}).
\end{enumerate}

The list above represents only a small portion of useful homogeneous spaces. Other examples can be found
in \cite{BD19}, \cite{BK16}, \cite{BST18}, etc.

It is clear from Definition \ref{baskdef26} that every homogeneous Banach space $\mathcal{X}$ is a non-degenerate Banach module $(\mathcal{X}, S)$ and, hence, a harmonious space according to Definition \ref{defharm}.
We also note that the terminology of homogeneous spaces is relatively new  and the corresponding definitions in  \cite{B13}, \cite{BD19}, and \cite{BST18} are slightly
different from Definition \ref{baskdef26}.
For example, Definition~\ref{baskdef26} is more general than \cite[Defintion 2.1]{B13}, and
 \cite[Definition 2.1]{BST18} requires injectivity of the embedding $\mathcal{X}\hookrightarrow S_1$ instead of the non-degeneracy Property 5.

Together with a homogeneous space $\mathcal{X}$, we consider the Sobolev-type space $\mathcal{X}^1$ which consists of all absolutely continuously functions $x\in\mathcal{X}$ with derivative in
$\mathcal{X}$.  We then have $D(A)\subset\mathcal{X}^1\subset\mathcal{X}$.

To apply the similarity theorems of the previous section, we must ensure that the potential function $v$ yields a multiplication operator $V$ in an $A$-admissible module. To achieve this, the function $v$ and the homogeneous space $\mathcal{X}$ must be somehow compatible. We use the following definition to describe the required compatibility.

\begin{definition}\label{baskdef27}
Let $\mathcal{F}$ be a Banach space of (equivalence classes of) of functions from $\mathbb{R}$ to $\mathrm{End}\,\mathscr{X}$ that is a Banach
 $L_1(\mathbb{R})$-module with the structure associated with the representation $S$.  Let also $\mathcal M$ be some $A$-harmonious space. The space
$\mathcal{F}$ is called $\mathcal{M}$-compatible if the map $v\mapsto V: (\mathcal{F}, S)\to (\mathcal{M}, \mathcal T_0)$
is an isometric homomorphism of Banach modules.
\end{definition}

For example, if  $\mathcal{X}=\mathcal{P}_\omega(\mathbb{R}, \mathscr{X})$ and $\mathcal M = \mathrm{End}\,\mathcal{X}$, then the space $\mathcal{F}=\mathcal{P}_\omega(\mathbb{R},
\mathrm{End}\,\mathscr{X})$ is $\mathcal{M}$-compatible. When  $\mathcal{X}=L_p(\mathbb{R}, \mathscr{X})$, $p\in [1,\infty)$, and $\mathcal M = \mathrm{End}\,\mathcal{X}$, an $\mathcal{M}$-compatible space $\mathcal{F}$
could be given by $L_\infty(\mathbb{R}, \mathrm{End}\,\mathscr{X})$, $\mathcal{P}_\omega(\mathbb{R},
\mathrm{End}\,\mathscr{X})$, or $\mathcal{AP}_\omega(\mathbb{R}, \mathrm{End}\,\mathscr{X})$. 

In the remainder of this section we assume that
$v\in\mathcal{F}\cap L_\infty(\mathbb{R}, \mathrm{End}\,\mathscr{X})$ for some $\mathcal{M}$-compatible space $\mathcal{F}$.

Our first similarity theorem for differential operators follows from Corollary \ref{maincor} of Theorem \ref{baskth4.2}.

\begin{theorem}\label{baskth6.1}
Assume that the operator $\mathfrak{L}=-i\frac d{dt}-V=A-V: D(A)\subset\mathcal{X}\to\mathcal{X}$ acts in a homogeneous space $\mathcal{X}$ and the operator $V$ from an  $A$-admissible space $\mathcal{M}$ is a multiplication operator by a potential function $v$ from an $\mathcal{M}$-compatible
space $\mathcal{F}$. Then the operator $\mathfrak{L}$ is similar to the operator $-i\frac d{dt}-V_0$, where $V_0\in\mathcal{M}$ is an operator of multiplication by a function
$v_0\in\mathcal{F}$, which is the restriction to $\mathbb R$ of an entire (operator-valued) function of exponential type.
\end{theorem}

\begin{remark}
It is important to note that Theorem \ref{baskth6.1} has no assumptions on the norm of the operator
 $V\in\mathcal{M}$. We also point out that the operator $V_0\in\mathcal{M}$ is of the form
 \[
 (V_0x)(t) = \int_{\mathbb R} \varphi_a(s)v_*(t-s)x(t)ds = (\varphi_a* v_*)(t)x(t) = v_0(t)x(t), \quad t\in\mathbb R, 
 \]
 where $a \in\mathbb R$ satisfies
 \[
 a > 13.64\|v\|_{\mathcal F},
 \]
 and the function  $v_*\in \mathcal F$ is the solution of the equation
\[
v_* = v(\psi_a* v_*) - (\psi_a* v_*)(\varphi_a* v_*)+v,
\]
with the functions  $\varphi_a$ and $\psi_a$ defined by \eqref{otrap} and \eqref{psia}, respectively.
 \end{remark}

In the next result, which is an immediate consequence of Theorem \ref{baskth5.1} and Corollary  \ref{baskth5.c1}, we consider a periodic case, in which a non-autonomous Cauchy problem is reduced to an autonomous one.

\begin{theorem}\label{baskth12}
Assume that for some $\omega > 0$ we have $\mathcal{X}=\mathcal{P}_\omega(\mathbb{R}, \mathscr{X})$,  $v\in \mathcal F =\mathcal{P}_\omega(\mathbb{R}, \mathrm{End}\,\mathscr{X})$, 
and
$$
\frac{5.4\sqrt{3}\omega}{\pi}\|v\|_{\mathcal F}<1.
$$
Then the operator $A-V$ is similar to the operator $A-V_0$, where $V_0$ is an operator of multiplication by a constant (operator-valued) function:
$(Vx)(t) = Cx(t)$, $C\in \mathrm{End}\,\mathscr{X}$.
\end{theorem}

We conclude this section by considering perturbations $V$ of the differential operator $A$ that belong to
$\mathscr {H\!\!C}_{2a}$ or $\mathscr {H\!\!A\!\!C}_{2a}$, $a>0$, i.e.~hypercausal or hyper-anticausal operators. Recall from Subsection \ref{spsub}  that since $V$ is an operator of multiplication by the function $v\in\mathcal{F}\cap L_\infty(\mathbb{R}, \mathrm{End}\,\mathscr{X})$, hypercausality (respectively, hyper-anticausality) means that the support of the Fourier transform of the function $v$ belongs to the interval $[2a, +\infty)$ (respectively, $(-\infty, -2a]$). From Theorem \ref{hypcaussim} we deduce the following remarkable result.

\begin{theorem}\label{baskth13}
Assume that $\supp \widehat v \subseteq [a, +\infty)$  or $\supp \widehat v \subseteq (-\infty, -a]$ for some $a>0$. Then the operator $A-V$ is similar to the operator $A$.
\end{theorem}

We remark that above we gave only a few examples of  $A$-admissible perturbation spaces for the operator $A=-i\frac d{dt}$.
In particular, some spaces in \cite{FSS14} and \cite{GK20} are also examples of $A$-admissible spaces.

\section{Application to operator matrices}\label{basksec6}

In this section, we apply the similarity theorems from Section \ref{basksec4} to operators that can be defined by matrices, either
scalar -- as in Example \ref{exa5} or operator -- as in Definitions \ref{baskdefnew26} and \ref{baskdefnew27} below, as well as in
\cite{B92,  B97Sib, B97Izv, BK14, BKU19}.

By $\mathbb{J}$ we shall denote one of the sets $\mathbb{N}$ or $\mathbb{Z}$. Given a complex Banach space $\mathcal{X}$, we consider a (disjunctive) resolution of the identity given by a family of idempotents $\mathscr P = \{P_k, k\in\mathbb{J}\}$. In particular, we have that $P_iP_j=0$ for $i\ne j$, $i$, $j\in\mathbb{J}$; the series $\sum_{k\in\mathbb{J}}P_kx$ unconditionally converges to $x\in\mathcal{X}$; and the equalities $P_kx=0$ for all $k\in\mathbb{J}$ imply that $x =0\in\mathcal{X}$.

From \cite[Theorem~3.10]{H11} we infer that unconditional convergence of the series
$
\sum P_kx,
$
allows us to define a strongly continuous representation  $T_{\mathscr{P}}: \mathbb{R}\to\mathrm{End}\,\mathcal{X}$,
by a formula analogous to \eqref{bask2.3P}. Without loss of generality (in view of Remark \ref{ogriz}), we assume that
the representation $T_{\mathscr{P}}: \mathbb{R}\to\mathrm{End}\,\mathcal{X}$ is isometric. The strong continuity of $T_{\mathscr{P}}$
implies that formula \eqref{basknew2} defines a Banach $L_1(\mathbb{R})$-module structure on $\mathcal{X}$ that is associated with $T_{\mathscr{P}}$.
As usual, the generator of the module $(\mathcal{X}, T_{\mathscr{P}})$ will be denoted by $A:D(A)\subset\mathcal{X}\to\mathcal{X}$
and will play the role of the unperturbed operator.

For example, if $\mathcal{X}=\ell_2(\mathbb{Z})$ and the resolution of the identity consists of the rank-one projections corresponding to the standard basis (see Example \ref{exa5}), then the representation $T_{\mathscr{P}}$ coincides with the modulation representation  given by
 \eqref{modl}. The generator $A$ in this case has a diagonal matrix (with respect to the standard basis of $\ell_2(\mathbb{Z})$), and its domain $D(A)$ consists of all vectors $x\in \ell_2(\mathbb{Z})$ such that $\sum_{n\in\mathbb{Z}}|nx_n|^2<\infty$.

In the following two definitions we disambiguate the notions of an operator matrix and a matrix of an operator.
 
\begin{definition}\label{baskdefnew26}
An operator matrix $X=(X_{ij})$ is defined as a map $X: \mathbb{J}\times\mathbb{J}\to\mathrm{End}\,\mathcal{X}$.
We say that an operator matrix $X=(X_{ij})$ is associated with the resolution of the identity $\{P_k, k\in\mathbb{J}\}$ if
$X_{ij}=P_iX_{ij}P_j$, $i$, $j\in\mathbb{J}$.
\end{definition}

\begin{definition}\label{baskdefnew27}
By the matrix of an operator  $X\in\mathrm{End}\,\mathcal{X}$ with respect to the resolution of the identity $\{P_k, k\in\mathbb{J}\}$
we mean the operator matrix defined by $X_{ij}=P_iXP_j$, $i$, $j\in\mathbb{J}$.
\end{definition}

In this section, we consider a fixed resolution of the identity $\mathscr P = \{P_k, k\in\mathbb{J}\}$ and operators are identified with their matrices with respect to it. The theorems in Section \ref{basksec4} immediately imply the following three results.

\begin{theorem}\label{mathrth1}
An operator $A-B$, with $B\in\mathrm{End}\,\mathcal{X}$, is similar to an operator $A-B_0$, where the matrix of operator $B_0\in\mathrm{End}\,\mathcal{X}$,
is banded, i.e.~has only finitely many non-zero diagonals.
\end{theorem}

\begin{theorem}\label{mathrth2}
Assume that $B\in\mathrm{End}\,\mathcal{X}$ and
$$
10.8\sqrt{3}\|B\|<1.
$$
Then the operator $A-B$ is similar to the operator $A-B_0$, whose matrix is diagonal.
\end{theorem}

\begin{theorem}\label{mathrth3}
Assume that the matrix of an operator $B\in\mathrm{End}\,\mathcal{X}$ is strictly upper (or lower) triangular.
Then the operator $A-B$ is similar to the operator $A$.
\end{theorem}

We remark that the operator $B$ in the above theorems can be chosen from other $A$-admissible spaces, not just from $\mathcal{M} = \mathcal{M}_0 = \mathrm{End}\,\mathcal{X}$. 
In some cases, for example when $\mathcal{M} = \mathfrak S_2(\mathcal{X})$ is the Hilbert-Schmidt ideal of operators in a Hilbert space $\mathcal X$, the constants in the condition on $\|B\|$ may be significantly improved \cite{BKU19}. 
We also note that the spaces considered in \cite{SSun19} are $A$-admissible.

The results of this section are significant in part because of the many approximate computational methods that are available for banded matrices \cite{GL13}. The similarity established in the above theorems allows one to use those methods to obtain results for matrices that are not initially banded.

We also note that analogs of Theorems \ref{mathrth1}, \ref{mathrth2}, and \ref{mathrth3} hold in the case when the representation comes from formula \eqref{bask2.3} rather than \eqref{bask2.3P}. In that case, the results most closely resemble those in
\cite{BDS11, BKU19, BP17}.
We conclude this section with a comparison of the versions of the method of similar operator used in this paper and in the ones just cited.

Let $\mathcal{X = H}$ be a complex Hilbert space and consider a self-adjoint operator $A: D(A)\subset\mathcal{H}\to\mathcal{H}$. Assume that $A$ has a discrete spectrum $\sigma(A)=\cup_{k\in\mathbb{Z}}\{\lambda_k\}$,
where $\lambda_k$, $k\in\mathbb{Z}$, are semi-simple eigenvalues of finite multiplicity. Let
$P_k=P(\{\lambda_k\}, A)$, $k\in\mathbb{Z}$, be the corresponding spectral projections that form a resolution of the identity. 
We have $AP_k=\lambda_kP_k$, $k\in\mathbb{Z}$. We also assume that the eigenvalues are separated:
\begin{equation}\label{bask7.3}
 d :=\inf_{i\ne j}|\lambda_i - \lambda_j| >0.
\end{equation}

The above conditions are common for all the versions of the method we are comparing (see Example \ref{exrepU} and \cite{BDS11, BKU19, BP17}).

A formula analogous to \eqref{bask2.3} defines a strongly continuous isometric representation  $\mathcal{T}=T_A$, which,
in turn, yields a Banach $L_1(\mathbb{R})$-module  $(\mathcal{H}, T_A)$ whose generator is
the operator $A$. As before, we will identify  operators $X\in\mathrm{End}\,\mathcal{X}$ with their matrices
$X\sim(X_{ij})$, where $X_{ij}=P_iXP_j$. The role of the $A$-admissible space
 $\mathcal{M}$ is once again played by $\mathrm{End}\,\mathcal{X}$.

\begin{lemma}
Given an operator $X\in\mathrm{End}\,\mathcal{H}$ and its matrix  $(X_{ij})$, $i, j\in\mathbb{Z}$, the elements of the matrix of the operator $\mathcal{T}_0(f)X$ are given by $(\widehat{f}(\lambda_i-\lambda_j)X_{ij})$.
\end{lemma}
\begin{proof}
We have $T_A(t)P_i=e^{i\lambda_it}P_i=P_iT_A(t)$ and $P_i(\mathcal{T}_0(t)X)P_i=e^{i(\lambda_i-\lambda_j)t}X_{ij}$. It follows that
$$
P_i(\mathcal{T}_0(f)X)P_i=\int_{\mathbb{R}}f(t)e^{-i(\lambda_i-\lambda_j)t}X_{ij}\,dt=\widehat{f}(\lambda_i-\lambda_j)X_{ij},
$$
and the lemma is proved.
\end{proof}

Consider the functions $\varphi_a$ and $\psi_a$ from \eqref{otrap} and \eqref{psia} with
\begin{equation}\label{bask*}
2a<d.
\end{equation}
Then the matrix of the operator $JX=\mathcal{T}_0(\varphi_a)X$ is defined by
$$
P_i(JX)P_j=
\begin{cases}
0, \quad &i\ne j, \\
X_{ij}, \quad &i=j,
\end{cases}
$$
and the matrix of the operator $\Gamma X=\mathcal{T}_0(\psi_a)X$ -- by
\begin{equation}\label{basknew7.4}
P_i(\Gamma X)P_j=
\begin{cases}
0, \quad &i=j, \\
\frac{X_{ij}}{\lambda_i-\lambda_j}, \quad &i\ne j.
\end{cases}
\end{equation}
We emphasize that the above formulas hold only under condition \eqref{bask*}.

Applying Theorem \ref{baskth5.1} to an operator $A-B$, with $B\in\mathrm{End}\,\mathcal{H}$, we get the following result.

\begin{theorem}
Assume that the spectrum $\sigma(A)$ satisfies conditions \eqref{bask7.3} and \eqref{bask*}. Assume also that the perturbation operator
$B\in \mathrm{End}\,\mathcal{H}$ satisfies
$$
\frac{5.4\sqrt{3}}{a}\|B\|<1.
$$
Then the operator $A-B: D(A)\subset\mathcal{H}\to\mathcal{H}$ is similar to 
$$
A-JX_*=A-\sum_{i\in\mathbb{Z}}P_iX_*P_i,
$$
where $X_*\in\mathrm{End}\,\mathcal{H}$ is the unique solution of the nonlinear equation \eqref{bask3.4}.
The intertwining operator of $A-B$ into $A-JX_*$ is $I+\Gamma X_*$, and its matrix elements are given by
 \eqref{basknew7.4} with $X^*$ in place of $X$. The operator $JX_*\in\mathrm{End}\,\mathcal{H}$
is memoryless and its matrix is diagonal.
\end{theorem}

Thus, the method of diagonalization of the operator $A-B$ employed in this paper coincides with that of \cite{BP17} and \cite{BKU19}. The same can be said about the diagonalization of Dirac operators in \cite{BDS11}. However, in \cite{BDS11, BKU19, BP17}, a different admissible triple was used.
Roughly speaking, in this paper the matrix of the operator $J$ is banded with the width of the band increasing as the number $a$ grows. In \cite{BDS11, BKU19, BP17}, the matrix of $J$ is block diagonal and only one (central) block grows in the situation analogous to the growth of $a$.
As a consequence, the transforms $\Gamma_a$ are also constructed differently.

\section{Additional remarks}\label{discus}
In this section, we gather a few supplementary remarks highlighting the contributions of this paper, as well as their importance and historical contextualization. 

This paper possesses a semi-expository character, primarily aiming to create and present a new modification of the method of similar operators. This modification is tailored for perturbed operators $A-B$, where the perturbation $B$ acts in a large class of Banach spaces, which we called $A$-admissible. The focal point of the paper lies in the development and elucidation of this approach, which has not been pursued in the literature before.

The most general of the presented results about $A$-admissible perturbations is Theorem \ref{baskth4.2}, which establishes similarity of the operator $A-B$ to the operator $A-C$ such that $C$ has a compact Beurling spectrum. The latter property of the operator $C$ allows one to study the spectral properties of the operator $A-B$ more effectively, as well as use computational methods that are not necessarily valid directly for $A-B$. 
It is important to reiterate that Theorem \ref{baskth4.2} stands apart by virtue of its lack of assumptions concerning the norm of the perturbation $B$ or the lacunary nature of the spectrum of $A$ -- prerequisites commonly encountered in the majority of existing literature.

From a historical perspective, a substantial portion of perturbation theory results were originally formulated within the context of Hilbert spaces, as opposed to the broader framework of Banach spaces. We reference  \cite{ DS88III, F65p, K76, MM82, MM84, MS17, SD79, Sh16, T65} as a brief glimpse into the vast array of existing research on this subject.

A significant portion of the results in Hilbert spaces   has been extended to Banach spaces, although often necessitating more stringent  assumptions on the perturbation operator $B$. In the context of similarity, these more stringent requirements frequently manifest through heightened constraints on the norm of $B$ or the space of admissible perturbations. When using the method of similar operators, the main reason for better constants in Hilbert spaces is the existence of better estimates on the norms of the transforms $J$ and $\Gamma$. For example, in most of the Hilbert space settings one has $\|J\|=1$. For a normal operator $A$ that has spectral components separated by $d$ as in \eqref{bask7.3}, one usually uses transforms $\Gamma$ that satisfy $\|\Gamma\| \le \frac5d$; when $A$ is self-adjoint, the estimate improves to $\|\Gamma\| \le \frac\pi{2d}$; if, additionally, the space of admissible perturbations is $\mathfrak S_2(\mathcal H)$, one gets $\|\Gamma\| = \frac1d$. These and other results concerning development of the method of similar operators in Hilbert spaces can be found in \cite[and references therein]{B94, B95, BDS11, BKU19, BP16, KS18}.

The method of similar operators, pioneered by Friedrichs \cite{F65p} and championed in this paper, is considerably less prevalent within the literature when juxtaposed with the resolvent method introduced by Kato \cite{K76}. The latter is based on the integral representation of the spectral projections corresponding to various components of the spectrum $\sigma(A)$. We note that the resolvent method is obviously inapplicable in the general setting of this paper, where the spectrum $\sigma(A)$ is connected. Furthermore, it is worth noting that the resolvent method particularly excels in Hilbert spaces, especially when the operator $B$ is normal or, even better, self-adjoint. In those narrower settings the resolvent method is likely to outperform the method of similar operators. However, the efficacy of the resolvent method within a Banach space framework is significantly curtailed, with a substantial portion of the technique becoming totally infeasible.  In stark contrast, the method of similar operators, which draws its roots from abstract harmonic analysis, remains commendably robust when applied within the context of Banach spaces.

\medskip
\centerline{\bf Acknowledgements}
\medskip

We are happy to dedicate this paper to our friend and colleague Akram Aldroubi on the occasion of his $65^{\mathrm {th}}$ birthday.

\bibliographystyle{abbrv}
\bibliography{refs}

\end{document}